%% file: ms.tex
\documentclass{amsart}

\newtheorem{theorem}{Theorem}[section] 
\newtheorem{lemma}[theorem]{Lemma}     
\newtheorem{corollary}[theorem]{Corollary}
\newtheorem{proposition}[theorem]{Proposition}
\newtheorem*{conjecture}{Conjecture}
\newtheorem*{unthm}{Theorem}
\theoremstyle{definition}
\newtheorem{definition}[theorem]{Definition}

\usepackage{amssymb}
\usepackage{amsmath}
\usepackage{mathrsfs}
\usepackage[nohug,heads = vee]{diagrams}
\usepackage{bbm}
\usepackage[cmtip,all]{xy}
\usepackage{enumerate}
\newcommand{\symgp}{\mathfrak{S}}
\newcommand{\hecke}{\mathscr{H}}
\newcommand{\hecku}{\underline{\mathscr{H}}}
\newcommand{\setx}{\underline{\mathscr{X}}^2}
\newcommand{\sety}{\underline{\mathscr{Y}}^2}
\newcommand{\setz}{\underline{\mathscr{Z}}^2}

\newcommand{\std}[1]{\mathop{\rm Std}(#1)}
\newcommand{\Endo}{\mathop{\rm End}\nolimits}
\newcommand{\End}[2]{\Endo_{#1}(#2)_{#1}}
\newcommand{\trace}{\mathop{\rm Tr}\nolimits}

\title[Vertices for Iwahori--Hecke algebras]%
	{Vertices for Iwahori--Hecke algebras and the Dipper--Du conjecture}

\author{James R. Whitley}
\address{School of Mathematics, University of Birmingham, Birmingham, B15 2TT, United Kingdom}
\email{jrw536@bham.ac.uk}

\dedicatory{Dedicated to the memory of Anton Evseev}

\thanks{The author is supported through the EPSRC}

\begin{document}

\begin{abstract}
Let $\hecke_n$ denote the Iwahori--Hecke algebra corresponding to the symmetric group $\symgp_n$. We set up a Green correspondence for bimodules of these Hecke algebras, and a Brauer correspondence between their blocks. We examine Specht modules for $\hecke_n$ and compute the vertex of certain Specht modules, before using this to give a complete classification of the vertices of blocks of $\hecke_n$ in any characteristic. Finally, we apply this classification to resolve the Dipper--Du conjecture about the structure of vertices of indecomposable $\hecke_n$-modules.
\end{abstract}

\maketitle

\input{section1}
\input{section2}
\input{section3}
\input{section4}
\input{section5}
\input{section6}
\input{section7}
\input{section8}
\input{section9}

\subsection*{Acknowledgements}\label{ackref}
I wish to thank Simon Goodwin for his insights and helpful comments while putting this paper together. In particular, I would like to thank my previous PhD supervisor Anton Evseev for all his guidance and support. Finally I am grateful to the referee for their helpful comments.

%
\end{document}

%% file: section1.tex
\section{Introduction}

Denote by $\symgp_n$ the symmetric group on $n$ letters generated by the elementary transpositions $s_i$, and define the Iwahori--Hecke algebra of type $A_{n-1}$ (henceforth just known as a Hecke algebra) in the following way. Let $F$ be an algebraically closed field of characteristic $p \geq 0$, pick $q\in F^\times$, and denote by $\hecke_n := \hecke_n(F, q)$ the associative algebra over $F$ generated by the set:
\[\{T_{i}: i = 1, \dots, n-1\}\] with relations:
\[(T_{i} - q)(T_{i}+1) = 0\text{ for }1\leq i \leq n-1,\]
\[T_{i}T_{j} = T_{j}T_{i}\text{ for }\lvert i-j\rvert  > 1,\]
\[T_{i}T_{i+1}T_{i} = T_{i+1}T_{i}T_{i+1}\text{for }1\leq i < n-1.\]
$\hecke_n$ has an $F$-basis $\{T_w : w\in \symgp_n\}$ (see for example~\cite[\S 1]{mathas}) where: 
\[T_w = T_{i_1}\cdots T_{i_t}\]
if $w = s_{i_1}\cdots s_{i_t}$ is a reduced expression for $w$. Under this convention we have $T_{s_i} = T_i$.

Let $e$ be the smallest integer such that $1 + q + \dots + q^{e-1} = 0$ if it exists, otherwise define $e = 0$. This is the \emph{quantum characteristic} of $\hecke_n$. In this paper, we will focus on the case where $e > 1$. If $p > 0$, then either $(e,p) := \text{hcf}(e,p) = 1$ and $q$ is a primitive $e$-th root of unity, or $e = p$ and $q = 1$. If $p=0$, then $e \neq 0$ means that $q$ is also an $e$-th root of unity. For more on the structure of $\hecke_n$, particularly its structure as a cellular algebra, see for example~\cite{mathas}. 

Relative projectivity and vertices of Hecke algebras were first introduced by Jones in~\cite{jones}, generalising the results from local representation theory of finite groups (see for example~\cite{alperin}). Let $\lambda$ be a composition of $n$ (writing $\lambda \vDash n$), with corresponding parabolic subgroup $\symgp_\lambda$ of $\symgp_n$, and parabolic subalgebra $\hecke_\lambda = \langle T_w : w\in \symgp_\lambda\rangle$ of $\hecke_n$. If $M$ is a $\hecke_n$-module, we say that $\symgp_\lambda$ is a \emph{vertex} of $M$ if $M$ is relatively $\hecke_\lambda$-projective, and if for any $\mu\vDash n$ with $M$ relatively $\hecke_\mu$-projective, then a conjugate of $\symgp_\lambda$ is a subgroup of $\symgp_\mu$. In~\cite{du}, Du gave a Green correspondence for modules of Hecke algebras, analagous to the classical correspondence in modular representation theory of finite groups. The main aim of this paper is to extend these notions to bimodules, and in particular to blocks, leading us to a Brauer correspondence.
\begin{unthm}[Brauer correspondence for Hecke algebras]
Let $n = a + de$, with $\mu = (a, de)$, $\tau = (1^a, de)$ and $\lambda = (1^a, \lambda_1, \dots, \lambda_s)$, where $(\lambda_1, \dots, \lambda_s)\vDash de$ and $\lambda_i \neq 1$ for all $i$. Then there is a one-to-one correspondence between blocks of $\hecke_\mu$ with vertex $(\symgp_\lambda,\symgp_\lambda)$ and blocks of $\hecke_n$ with the same vertex. 
\end{unthm}
Given this, we are able to explictly compute vertices of the blocks of $\hecke_n$, by identifying the vertex of a block of the right $\hecke_\mu$, and identifying its Brauer correspondent. To do this, we need the following definitions. Given $n\in \mathbb{N}$, write $n$ as its \emph{$e$-$p$-adic expansion} by:
\[n = a_{-1} + a_0e + a_1ep + \dots a_tep^t\] 
where $0\leq a_{-1}< e$ and $0\leq a_i < p$, for $i \geq 0$. If $n$ has the above $e$-$p$-adic expansion, the \emph{standard maximal $e$-$p$-parabolic subgroup} of $\symgp_n$ is the subgroup corresponding to the composition: 
\[(1^{a_{-1}}, e^{a_0}, (ep)^{a_1}, \dots, (ep^t)^{a_t})\vDash n.\]
A general \emph{$e$-$p$-parabolic subgroup} of $\symgp_n$ corresponds to a composition $\tau = (\tau_1, \dots, \tau_s)$ of $n$ which has for each $i$, $\tau_i = 1$ or $\tau_i = ep^{r_i}$ for some $r_i\geq 0$.

By Nakayama's Conjecture (see for instance~\cite[Corollary 5.38]{mathas}), we can label the blocks of $\hecke_n$ by $e$-cores and $e$-weights. Using these definitions, we can state the main result of this paper.

\begin{unthm}[Classification of vertices of blocks of Hecke algebras]
Let $F$ be an algebraically closed field, $q\neq 0 \in F$ with quantum characteristic $e \neq 0$, and $B = B_{\rho, d}$ the block of $\hecke_n := \hecke_n(F,q)$ corresponding to the $e$-core $\rho$ and $e$-weight $d$ (so in particular $n = |\rho| + ed$). If $d = 0$, then $B$ is a projective $(\hecke_n, \hecke_n)$-bimodule. Otherwise, let $\tau = (\tau_1, \dots, \tau_s)$ be the composition corresponding to the $e$-$p$-adic expansion of $de$, and define $\lambda = (1^{|\rho|},\tau_1, \dots, \tau_s)$. Then the vertex of $B$ as a $(\hecke_n, \hecke_n)$-bimodule is $(\symgp_\lambda, \symgp_\lambda)$. 
\end{unthm}

In~\cite{dipperdu}, Dipper and Du showed that for trivial and alternating source modules of $\hecke_n$, the vertex will always be an $e$-$p$-parabolic subgroup, and conjectured that this should hold for any $\hecke_n$-modules. This was shown to be true if $p = 0$ in~\cite{du}, and proved for blocks of finite representation type (i.e. by~\cite[Theorem 1.2]{erdmann} blocks of $e$-weight 1) in~\cite{schmider}. As a corollary to the previous theorem, we are able to resolve this conjecture, proving:

\begin{conjecture}[Dipper--Du]
\label{dduconj}
Let $F$ be an (algebraically closed) field of characteristic $p > 0$, $n\in \mathbb{N}$, and $q\neq 0\in F$ with quantum characteristic $e > 0$. Then the vertices of indecomposable $\hecke_n(F,q)$-modules are $e$-$p$-parabolic.
\end{conjecture}

The paper is structured as follows. The next section introduces the notation we will use. In Section 3, relative projectivity for bimodules is explored and a Green correspondence is given, before we expand it to a Brauer correspondence in Section 4, and give a method of identifying these correspondents. Section 5 looks at indecomposability and restrictions of Specht modules for $\hecke_n$, and Sections 6 and 7 look at the vertex of the sign module, and relative projectivity of blocks, in cases where the characteristic of $F$ is both zero and prime. Our classification is proved in Section 8, before finally using it to resolve the Dipper--Du conjecture in Section 9. 

%% file: section2.tex
\section{Preliminaries}

Note that all modules we will be using are right modules unless stated otherwise.

\subsection{Partitions and parabolics}

Here we will briefly recap the combinatorics relating to partitions and Young tableaux. We will use the notation and conventions from~\cite[\S 3]{mathas}. We say that $\lambda$ is a \emph{composition} of $n$, and write $\lambda \vDash n$, if $\lambda = (\lambda_1,\dots,\lambda_s)$ is a tuple of positive integers with $\sum_{i=1}^s \lambda_i = n$.  We say that $\lambda$ is a \emph{partition} of $n$ (and write $\lambda \vdash n$) if $\lambda$ is a composition of $n$, and $\lambda_i \geq \lambda_{i+1}$ for each admissible $i$. We denote the unique partition of $0$ by $\emptyset$. For $\lambda = (\lambda_1, \dots, \lambda_s)\vDash n$, we define the \emph{parabolic subgroup} of $\symgp_n$ corresponding to $\lambda$ as follows:
\begin{align*}\symgp_\lambda &:= \symgp_{\{1,\dots,\lambda_1\}} \times \symgp_{\{\lambda_1+1, \dots, \lambda_1+\lambda_2\}} \times \cdots \times \symgp_{\{\left(\sum_{i=1}^{s-1} \lambda_i\right) +1, \dots , \sum_{i=1}^s \lambda_s\}}\\
& \,\,\cong \symgp_{\lambda_1}\times \cdots \times \symgp_{\lambda_s}.
\end{align*}
Sometimes a more general definition of parabolic subgroup is given, however, as we are only interested in these subgroups up to conjugation in $\symgp_n$, this definition suffices for our purposes. Similarly, we can define a \emph{parabolic subalgebra} $\hecke_\lambda$ of $\hecke_n$ as the following $F$-span:
\[\hecke_\lambda = \langle T_w: w\in \symgp_\lambda\rangle.\]
Note that we can implicitly identify $\hecke_\lambda$ with $\hecke_{\lambda_1}\otimes \cdots\otimes \hecke_{\lambda_s}$, the $s$-fold tensor product over $F$ in the following way. Let $T_{j}$ be a generator of $\hecke_\lambda$ with $j = \sum_{i=1}^{k-1} \lambda_i + l$, for $1 \leq l < \lambda_{k}$. Then our map identifies $T_{j}$ with the following simple tensor:
\[1 \otimes\cdots \otimes 1\otimes T_{l}\otimes 1 \otimes \cdots\otimes 1,\]
where $T_{l}$ lies in the $k$-th part of the tensor product. We do this implicitly throughout this paper.

For a partition $\lambda$, we can also form its corresponding Young diagram, and fill in the boxes using the numbers $1, \dots, n$ exactly once to get a Young tableau. For more about these, see~\cite[\S3.1]{mathas}, or~\cite[\S 2.7]{jameskerber}. We say a tableau is standard if the entries are increasing along all rows and down all columns, and denote the set of standard $\lambda-$tableaux by $\std{\lambda}$. In particular, we will use the notation $\mathfrak{t}^\lambda$ to denote the standard tableau where the numbers $1, \dots, n$ are placed in increasing order, first along the top row of the tableau of shape $\lambda$, then the second row, etc.

We also require the concepts of $a$-cores and $a$-hooks; for some $a > 0$. Given $\lambda\vdash n$, a \emph{$a$-hook} is a chain of boxes of length $a$ that can be removed from the rim of a diagram of shape $\lambda$ to get a diagram of shape $\rho$ where $\rho \vdash n-a$. The \emph{$a$-core} of $\lambda$ is the partition associated to the diagram gained from the diagram of shape $\lambda$ by recursively removing as many $a$-hooks as possible. This is uniquely determined, for example see~\cite[Theorem 2.7.16]{jameskerber}. Finally, the \emph{$a$-weight} of a partition, is the number of $a$-hooks you need to remove to reach its $a$-core. 

By Nakayama's Conjecture (as stated in~\cite[Theorem 6.1.21]{jameskerber}), the blocks of the group algebra $F\symgp_n$ can be parameterised by $p$-cores and $p$-weights, where $p$ is the characteristic of $F$. This includes the case where $p = 0$, where every partition is a zero-core, and thus lies in its own block.

Similarly, when $e > 0$, by~\cite[Corollary 5.38]{mathas}, the blocks of $\hecke_n$ can be labelled by $e$-hooks and $e$-weights, and we denote the block of $\hecke_n$ with $e$-core $\rho$ and $e$-weight $d$ by $B_{\rho, d}$.

\subsection{Coset representatives for $\symgp_n$}

Let $\sigma, \lambda, \nu \vDash n$ be compositions with both $\symgp_\lambda \subseteq \symgp_\sigma$, and $\symgp_\nu \subseteq \symgp_\sigma$. Denote by $\mathscr{R}_{\lambda}^\sigma$ the set of \emph{minimal right coset representatives} of $\symgp_{\lambda}$ in $\symgp_\sigma$, denote by $\mathscr{L}_{\lambda}^\sigma$ the set of \emph{minimal left coset representatives} of $\symgp_{\lambda}$ in $\symgp_\sigma$, and denote $\mathscr{D}_{\lambda, \nu}^\sigma$ to be the set of \emph{minimal double coset representatives} of $\symgp_\lambda$ and $\symgp_\nu$ in $\symgp_\sigma$. Note that by a minimal coset representative, we mean the unique element in that coset which is shortest with respect to the usual length function $\ell$ on $\symgp_n$. Properties of these can be found in~\cite[\S3, \S4]{mathas}. As a consequence of these properties, we can determine specific double coset representatives, as stated in the following lemma.
\begin{lemma} 
\label{mu-mu_cosets}
Let $\mu = (a,b) \vDash a+b = n$. Then:
\[\mathscr{D}^{(n)}_{\mu,\mu}=\left\{d_k = \prod_{i=1}^k (a-k+i,a+i) \,\bigg|\, k=0,\dots,\min(a,b)\right\}.\]
\end{lemma}

%% file: section3.tex
\section{Relative projectivity and the Green correspondence for bimodules}

Let $A$ be an $F$-algebra with subalgebra $A'\subseteq A$. Recall that an $A$-module $M$ is \emph{relatively $A$-projective} (or just $A$-projective), if for any $A$-modules $V$ and $W$ with $A$-algebra maps $\alpha$ and $\beta$ making the below diagram exact, the existence of an $A'$ map from $M$ to $V$ making the diagram commute, implies there is also an $A$ map from $M$ to $V$ making the diagram commute.
\begin{diagram}
	&	&	M & &\\
	& \ldDashto & \dTo_{\alpha} & &\\
V 	& \rOnto_{\beta} & W & \rTo & 0
\end{diagram}
Note that if we take $A'=F$, we obtain our usual notion of projectivity. A more practical definition of relative projectivity for $A$-modules is given by Higman's criterion (see for example~\cite[Theorem 2.34]{jones} for the Hecke algebra version) as stated below. For two modules $M$ and $N$, we use the notation $ M \mid N$ to say that $M$ is isomorphic to a direct summand of $N$.
\begin{theorem}[Higman's criterion] 
\label{higman} 
Let $A'\subseteq A$ be $F$-algebras, and let $M$ be a right $A$-module. Then the following are equivalent:
\begin{enumerate}[(a)]
\item $M$ is $A'$-projective,
\item $M\mid M\otimes_{A'} A$,
\item $M\mid U \otimes_{A'} A$ for some $A'$-module $U$,
\end{enumerate}
\end{theorem}
We have the following corollaries. First of all, by the second and third criteria, it is clear that if we have $A''\subseteq A' \subseteq A$, and $M$ is an $A''$-projective $A$-module, then it is also an $A'$-projective $A$-module. Similarly we have:
\begin{corollary}
\label{transitivity}
Let $A''\subseteq A'\subseteq A$ be $F$-algebras. Then for an $A$-module $M$, if $M$ is relatively $A'$-projective as an $A$-module, and relatively $A''$-projective as an $A'$-module, then $M$ is relatively $A''$-projective as an $A$-module.
\end{corollary}
We also have the following corollary about how relative projectivity behaves when tensoring two modules over $F$.
\begin{corollary}
\label{tensors}
Let $A'\subseteq A$ and $B'\subseteq B$ be $F$-algebras, $M$ an $A'$-projective $A$-module, and $N$ a $B'$-projective $B$-module. Then $M\otimes N$ is $A'\otimes B'$-projective as an $A\otimes B$-module.  
\end{corollary}
\begin{proof}
By Higman's criterion, $M \mid M\otimes_{A'}A$ and $N\mid N \otimes_{B'}B$. Therefore tensoring together over $F$ gives us as $A\otimes B$-modules:
\[M \otimes N \mid (M\otimes_{A'} A) \otimes (N \otimes_{B'}B).\]
It is straightforward to verify that the natural map $\varphi$ defined on pure tensors as
\[(m\otimes a)\otimes (n\otimes b) \mapsto (m\otimes n) \otimes (a\otimes b)\]
for $m\in M$, $n\in N$, $a\in A$ and $b\in B$ gives an $A\otimes B$-module isomorphism
\[\varphi: (M\otimes_{A'} A)\otimes (N\otimes_{B'} B) \to (M\otimes N) \otimes_{A'\otimes B'} (A\otimes B).\]
As such, we can conclude by Higman's criterion.
\end{proof}

In~\cite[Theorem 2.29]{jones}, a Mackey formula for Hecke algebras was given, and as a consequence, Jones was able to make concrete the notion of a \emph{vertex} of a $\hecke_n$ module~\cite[Theorem 2.35]{jones}. For a $\hecke_n$-module $M$, this is a parabolic subgroup $\symgp_\lambda$ (for some $\lambda \vDash n$) such that $M$ is $\hecke_\lambda$-projective, and for any $\mu \vDash n$, if $M$ is $\hecke_\mu$-projective, then a conjugate of $\symgp_\lambda$ is contained in $\symgp_\mu$. This is not unique, but it is determined up to conjugation in $\symgp_n$.

Combining the notion of a vertex with our previous corollary, we can show that the vertex of a module also behaves as one would expect when taking tensor products. For the rest of this section, we will be working with $\hecke_{\sigma}$-modules where $\sigma\vDash n$, instead of $\hecke_n$-modules. All definitions and results carry across in the same way, and this helps us work in more generality later on when doing our inductive arguments. We will also in future say that a module $M$ is $\symgp_\lambda$-projective instead of $\hecke_\lambda$-projective to mirror the notation used in~\cite{alperin}.

\begin{theorem}
\label{vertex_tensor}
Let $\tau_1, \sigma_1\vDash n$ and $\tau_2, \sigma_2\vDash m$, with $\symgp_{\tau_i}\subseteq \symgp_{\sigma_i}$ for $i=1,2$. If $M$ is a $\hecke_{\sigma_1}$-module with vertex $\symgp_{\tau_1}$, and $N$ is a $\hecke_{\sigma_2}$-module with vertex $\symgp_{\tau_2}$, then $M\otimes N$ has vertex $(\symgp_{\tau_1}\times\symgp_{\tau_2})$ as a $\hecke_{\sigma_1}\otimes \hecke_{\sigma_2}$-module.
\end{theorem}
\begin{proof}
By Corollary~\ref{tensors}, $M\otimes N$ is $(\symgp_{\tau_1}\times \symgp_{\tau_2})$-projective as a $\hecke_{\sigma_1}\otimes \hecke_{\sigma_2}$-module. Suppose that $\symgp_{\lambda_1}\times \symgp_{\lambda_2}$ is a vertex of $M\otimes N$ as a $\hecke_{\sigma_1}\otimes \hecke_{\sigma_2}$-module. Thus a $(\symgp_{\sigma_1}\times \symgp_{\sigma_2})$-conjugate of $\symgp_{\lambda_1}\times \symgp_{\lambda_2}$ is contained in $\symgp_{\tau_1}\times \symgp_{\tau_2}$. As a $\hecke_{\sigma_1}$-module, $M\otimes N$ is $\symgp_{\lambda_1}$-projective since:
\[M\otimes N \mid (M\otimes N)\otimes_{\hecke_{\lambda_1}\otimes \hecke_{\sigma_2}} \hecke_{\sigma_1}\otimes \hecke_{\sigma_2} \cong (M\otimes N)\otimes_{\hecke_{\lambda_1}}\hecke_{\sigma_1},\]
as $\hecke_{\sigma_1}$-modules, as $\hecke_{\sigma_1}$ only acts on the part induced from $M$. Here we used the fact that $M\otimes N$ is $(\symgp_{\lambda_1}\times\symgp_{\sigma_2})$-projective as $\symgp_{\lambda_2}\subseteq \symgp_{\sigma_2}$.

Furthermore, as a $\hecke_{\lambda_1}$-module, $M\otimes N\cong M^{\oplus \dim N}$, and thus $M$ too is $\symgp_{\lambda_1}$-projective as an $\hecke_{\sigma_1}$-module. So, some conjugate of $\symgp_{\tau_1}$ is contained in $\symgp_{\lambda_1}$ as $M$ has vertex $\symgp_{\tau_1}$. As we already know that a $\symgp_{\sigma_1}$ conjugate of $\symgp_{\lambda_1}$ is contained in $\symgp_{\tau_1}$, we conclude that $\symgp_{\lambda_1}$ is a conjugate of $\symgp_{\tau_1}$. 

Repeating on the other side with $N$, gives us that $\symgp_{\lambda_i}$ is a conjugate of $\symgp_{\tau_i}$ for $i=1,2$, and hence $(\symgp_{\tau_1}\times \symgp_{\tau_2})$ is a vertex of $M\otimes N$ as a $\hecke_{\sigma_1}\otimes\hecke_{\sigma_2}$-module.
\end{proof}

\subsection{Relative projectivity of bimodules}

Let $A, B$ be $F$-algebras with subalgebras $A'\subseteq A$ and $B'\subseteq B$. Then an $(A,B)$-bimodule is the same as a left $A\otimes B^{\text{op}}$-module. Hence we will say that an $(A,B)$-bimodule is relatively $(A', B')$-projective if as a left $A\otimes B^{\text{op}}$-module, $M$ is relatively $A'\otimes (B')^{\text{op}}$-projective. 

Using this, we can extend Higman's criterion and its corollaries to bimodules of Hecke algebras. Let $\sigma_1, \sigma_2 \vDash n$, and denote $\hecku_{\sigma_1, \sigma_2} := \hecke_{\sigma_1}\otimes \hecke_{\sigma_2}^{\text{op}}$, so a $(\hecke_{\sigma_1}, \hecke_{\sigma_2})$-bimodule is the same as a left $\hecku_{\sigma_1, \sigma_2}$-module. Finally use $\underline{T}_{w_1, w_2}$ to denote $T_{w_1}\otimes T_{w_2} \in \hecku_{\sigma_1, \sigma_2}$, for $w_i\in \symgp_{\sigma_i}$. Note that under this notation if we have a $(\hecke_{\lambda_1}, \hecke_{\lambda_2})$-bimodule $M$, then
\[\hecku_{\sigma_1, \sigma_2} \otimes_{\hecku_{\lambda_1, \lambda_2}} M \cong \hecke_{\sigma_1}\otimes_{\hecke_{\lambda_1}} M\otimes_{\hecke_{\lambda_2}} \hecke_{\sigma_2}\]
as $(\hecke_{\sigma_1}, \hecke_{\sigma_2})$-bimodules. This can be seen either using the transitivity of induction, or by the associativity formula given in~\cite[\S 9, Proposition 2.1]{cartan}. This gives a useful result if our bimodule is a block of $\hecke_n$.
\begin{proposition} 
\label{argue_down}
Let $B$ be a direct summand of $\hecke_n$ as a $(\hecke_n, \hecke_n)$-bimodule, which is $(\symgp_{\tau}, \symgp_n)$-projective. Then $B$ is $(\symgp_\tau, \symgp_\tau)$-projective.
\end{proposition}
\begin{proof}
By Higman's criterion, $B \mid \hecke_n\otimes_{\hecke_\tau} B \otimes_{\hecke_n} \hecke_n$. Since $B$ is a direct summand of $\hecke_n$ as a $(\hecke_n, \hecke_n)$-bimodule, it is also a direct summand of $\hecke_n$ as a $(\hecke_\tau, \hecke_n)$-bimodule. Hence,
\[ B\mid \hecke_n \otimes_{\hecke_\tau} \hecke_n \otimes_{\hecke_n} \hecke_n \cong \hecke_n\otimes_{\hecke_\tau}\hecke_n \cong \hecke_n\otimes_{\hecke_\tau} \hecke_\tau \otimes_{\hecke_\tau} \hecke_n,\]
thus by Higman's criterion again, $B$ is $(\symgp_\tau, \symgp_\tau)$-projective.
\end{proof}

As $\hecke_n$ (and thus $\hecke_{\sigma}$ for any $\sigma \vDash n$) has an anti--automorphism given by $T_w \mapsto T_{w^{-1}}$ for $w\in \symgp_n$ (see for instance~\cite[\S 3.2]{mathas}),
\[\hecke_{\sigma_1}\otimes \hecke_{\sigma_2}^{\text{op}} \cong \hecke_{\sigma_1}\otimes \hecke_{\sigma_2} \cong \hecke_\sigma\] 
as $F$-algebras, where $\sigma \vDash 2n$ is given by the concatenation of $\sigma_1$ and $\sigma_2$. Thus we can conclude from~\cite[Theorem 2.29]{jones} a Mackey formula for bimodules.

\begin{theorem}[Mackey formula for bimodules]
\label{Mackey}
For $i=1,2$, let $\symgp_{\lambda_i}, \symgp_{\mu_i}$ be parabolic subgroups of $\symgp_{\sigma_i}$, and denote $\mathscr{D}_i = \mathscr{D}_{\lambda_i,\mu_i}^{\sigma_i}$. Then for any left $\hecku_{\lambda_1,\lambda_2}$-module $N$, we have that as $\hecku_{\mu_1, \mu_2}$-modules:
\[\hecku_{\sigma_1, \sigma_2}\otimes_{\hecku_{\lambda_1,\lambda_2}} N\cong \bigoplus_{d_1\in \mathscr{D}_1,d_2\in \mathscr{D}_2} \hecku_{\mu_1,\mu_2} \otimes_{\hecku_{\nu(d_1),\nu(d_2)}}\left( \underline{T}_{d_1^{-1},d_2}\otimes_{\hecku_{\lambda_1,\lambda_2}}N\right)\]
where $\nu(d_i)\vDash n$ is defined via:
\[\symgp_{\nu(d_i)} = \symgp_{\lambda_i}^{d_i}\cap \symgp_{\mu_i}\] for $i = 1,2$.
\end{theorem}

Note that in this statement $\left( \underline{T}_{d_1^{-1},d_2}\otimes_{\hecku_{\lambda_1,\lambda_2}}N\right)\cong T_{d_1^{-1}} \otimes_{\hecke_{\lambda_1}} N\otimes_{\hecke_{\lambda_2}} T_{d_2}$ is indeed a $(\hecke_{\nu(d_1)}, \hecke_{\nu(d_2)})$-bimodule. To see this, let $w\in \symgp_{\nu(d_1)}$, $n\in N$, then:
\[T_wT_{d_1^{-1}}\otimes n \otimes T_{d_2} = T_{wd_1^{-1}} \otimes n \otimes T_{d_2}\] since $d_1^{-1}$ is a minimal right coset representative for $\symgp_{\mu_1}$. Since, $d_1wd_1^{-1}\in \symgp_{\lambda_1}$, and as $d_1^{-1}$ is a minimal left coset representative for $\symgp_\lambda$, we have that $T_{wd_1^{-1}} = T_{d_1^{-1}d_1wd_1^{-1}} = T_{d_1^{-1}}T_{d_1wd_1^{-1}}$. Thus we can pull $T_{d_1wd_1^{-1}}$ across the tensor product to $N$. Doing something similar on the right confirms our claim.

As before, using again the fact that $\hecke_{\sigma_i}$ possesses an anti--automorphism, as a consequence of~\cite[Theorem 2.31]{jones}, we can define a vertex of a $(\hecke_{\sigma_1}, \hecke_{\sigma_2})$-bimodule.

\begin{theorem}
\label{vertices}
Let $M$ be a $(\hecke_{\sigma_1}, \hecke_{\sigma_2})$-bimodule. Then there exist a pair of parabolic subgroups $\symgp_{\lambda_i}\subseteq \symgp_{\sigma_i}$ for $i=1,2$, such that $M$ is relatively $(\symgp_{\lambda_1},\symgp_{\lambda_2})$-projective and if for any parabolic subgroups $\symgp_{\tau_i}\subseteq \symgp_{\sigma_i}$ with $M$ relatively $(\symgp_{\tau_1},\symgp_{\tau_2})$-projective, then there is $x_i\in \symgp_{\sigma_i}$ with $\symgp_{\lambda_i}^{x_i}\subseteq \symgp_{\tau_i}$, again for $i=1,2$. We call the pair $(\symgp_{\lambda_1}, \symgp_{\lambda_2})$ a \emph{vertex} of $M$ as a $(\hecke_{\sigma_1}, \hecke_{\sigma_2})$-bimodule.
\end{theorem}

Using this definition, we get the following consequences of~\cite[Lemma 3.2]{du}.

\begin{lemma}\label{3.2i} \label{3.2ii}
Let $M$ be an indecomposable $(\hecke_{\sigma_1},\hecke_{\sigma_2})$-bimodule with vertex $(\symgp_{\tau_1},\symgp_{\tau_2})$ for $\tau_1, \tau_2\vDash n$, and let $\lambda_1, \lambda_2\vDash n$ with $\symgp_{\tau_i}\subseteq\symgp_{\lambda_i}\subseteq\symgp_{\sigma_i}$ for $i=1,2$. Then there are indecomposable $(\hecke_{\lambda_1},\hecke_{\lambda_2})$-bimodules $P$ and $Q$, both with vertex $(\symgp_{\tau_1}, \symgp_{\tau_2})$ such that:
\begin{enumerate}[(a)]
\item $P\mid M$ as $(\hecke_{\lambda_1}, \hecke_{\lambda_2})$-bimodules,
\item $M\mid\hecku_{\sigma_1, \sigma_2}\otimes_{\hecku_{\lambda_1,\lambda_2}} Q$.
\end{enumerate}
\end{lemma}
Note that in this situation, $Q$ corresponds to the notion of a \emph{source} for $M$ (see for example~\cite[\S9]{alperin}). The final lemma we state in this section is a consequence of~\cite[Lemma 3.3]{du} using Theorem~\ref{Mackey}.
\begin{lemma}
\label{3.3}
Let $\tau_i$, $\lambda_i$, $\sigma_i$ be as in Lemma~\ref{3.2i} for $i=1,2$. If $N$ is a $(\symgp_{\tau_1},\symgp_{\tau_2})$-projective $(\hecke_{\lambda_1},\hecke_{\lambda_2})$-bimodule, then we get as $(\hecke_{\lambda_1}, \hecke_{\lambda_2})$-bimodules:
\[\hecku_{\sigma_1, \sigma_2}\otimes_{\hecku_{\lambda_1,\lambda_2}} N \cong N \oplus Y,\]
where each indecomposable summand of $Y$ has a vertex contained in:
\[(\symgp_{\tau_1}^{d_1}\cap \symgp_{\lambda_1},\symgp_{\tau_2}^{d_2}\cap \symgp_{\lambda_2})\] for some $d_i\in \mathscr{D}_{\tau_i,\lambda_i}^{\sigma_i}$ with $(d_1,d_2)\neq (1,1)$.
\end{lemma}

\subsection{A Green correspondence for bimodules}

In this section, we hope to achieve a Green correspondence for our bimodules, as in~\cite[\S 3]{du}, or as done in~\cite[\S 11]{alperin} for finite groups. Let us fix some notation. Let $\lambda_i, \mu_i, \sigma_i$ be compositions of $n$ for $i=1,2$, with: 
\begin{equation}
\label{key_green}
\symgp_{\lambda_i}\subseteq N_{\symgp_{\sigma_i}}(\symgp_{\lambda_i})\subseteq \symgp_{\mu_i}\subseteq \symgp_{\sigma_i}.\end{equation}
Denote the following set:
\[\mathscr{P} = \{(H_1, H_2): H_i\text{ is a parabolic subgroup of }\symgp_{\sigma_i}\text{ for }i=1,2\}.\] 
For any subset $\mathscr{J}\subseteq \mathscr{P}$, we say a $(\hecke_{\sigma_1}, \hecke_{\sigma_2})$-bimodule is \emph{relatively $\mathscr{J}$-projective} (or just $\mathscr{J}$-projective), if each of its indecomposable summands is projective for some pair of parabolic subgroups in $\mathscr{J}$. Let $(P_1, P_2), (G_1, G_2)\in \mathscr{P}$. Then say that $(P_1,P_2)\in_{G_1,G_2}\mathscr{J}$ if there are elements $x_i\in G_i$ with $(P_1^{x_1},P_2^{x_2})\in \mathscr{J}$. Now we are ready to define the sets used in our version of the Green correspondence. 
\[\setx = \{ (H_1, H_2)\in \mathscr{P} : H_i\subseteq \symgp_{\lambda_i}^{d_i}\cap \symgp_{\lambda_i}\text{ for }(d_1, d_2)\in (\symgp_{\sigma_1},\symgp_{\sigma_2}) - (\symgp_{\mu_1}, \symgp_{\mu_2})\}\]
\[\sety = \{ (H_1, H_2)\in \mathscr{P} : H_i\subseteq \symgp_{\lambda_i}^{d_i}\cap \symgp_{\mu_i}\text{ for }(d_1, d_2)\in (\symgp_{\sigma_1},\symgp_{\sigma_2}) - (\symgp_{\mu_1}, \symgp_{\mu_2})\}\]
\[\setz = \{H = (H_1,H_2)\in \mathscr{P}: H_1\subseteq \symgp_{\lambda_1}, H_2 \subseteq \symgp_{\lambda_2}, H\notin_{\symgp_{\sigma_1},\symgp_{\sigma_2}} \setx\}.\]
Note that in the definitions of $\setx$ and $\sety$, we require that $d = (d_1,d_2)$ cannot have both $d_1\in \symgp_{\mu_1}$ and $d_2 \in \symgp_{\mu_2}$, but for example we could have $d_1 \in \symgp_{\mu_1}$ as long as $d_2\notin \symgp_{\mu_2}$. This follows from Lemma~\ref{3.3}, where at most one of the $d_i$ in that formula can be the identity. These sets are bimodule analogues of the sets used in both~\cite[\S 11]{alperin} and~\cite[\S 3]{du}. As in both the classical Green correspondence and in~\cite[Theorem 3.6]{du}, we have the following conditions linking our sets.
\begin{lemma}
\label{3.4II}
If $\symgp_{\tau_i}\subseteq \symgp_{\lambda_i}$ are parabolic subgroups for $i = 1,2$, then the following are equivalent:
\begin{enumerate}[(a)]
\item $(\symgp_{\tau_1},\symgp_{\tau_2})\in_{(\symgp_{\sigma_1},\symgp_{\sigma_2})} \underline{\mathscr{X}}^2$
\item $(\symgp_{\tau_1},\symgp_{\tau_2})\in \underline{\mathscr{X}}^2$
\item $(\symgp_{\tau_1},\symgp_{\tau_2})\in \underline{\mathscr{Y}}^2$
\item $(\symgp_{\tau_1},\symgp_{\tau_2})\in_{(\symgp_{\mu_1},\symgp_{\mu_2})} \underline{\mathscr{Y}}^2$
\end{enumerate}
\end{lemma}
Again, this follows as a consequence of~\cite[Lemma 3.4]{du}. Alternatively, it can be seen by the fact that $(H_1,H_2)\in \setx$ if and only if one of $H_1$ or $H_2$ lies in the corresponding set $\mathscr{X}$ from~\cite[\S3]{du}. We now need the following corollary, which corresponds to~\cite[Corollary 3.5]{du}.
\begin{corollary}
\label{3.5}
If $M$ is a $\setx$-projective $(\hecke_{\sigma_1}, \hecke_{\sigma_2})$-bimodule, then as a $(\hecke_{\mu_1}, \hecke_{\mu_2})$-bimodule, $M$ is $\sety$-projective.
\end{corollary}
\begin{proof}
Let $L$ be an indecomposable summand of $M$ as a $(\hecke_{\sigma_1},\hecke_{\sigma_2})$-bimodule, with vertex $(\symgp_{\tau_1},\symgp_{\tau_2})\in\setx$. Thus $L\mid \hecku_{\sigma_1, \sigma_2} \otimes_{\hecke_{\tau_1,\tau_2}} L$, and applying our Mackey Formula says that as a $\hecku_{\mu_1,\mu_2}$-module, each indecomposable summand of $L$ is $(\symgp_{\gamma_1}, \symgp_{\gamma_2})$-projective, where $\symgp_{\gamma_i}\subseteq \symgp_{\tau}^{d_i}\cap \symgp_{\mu_i}$ for some $d_i\in \mathscr{D}_{\tau_i, \mu_i}^{\sigma_i}$.

If both $d_i \neq 1$, then for $i = 1,2$:
\[\symgp_{\gamma_i}\subseteq \symgp_{\tau_i}^{d_i}\cap \symgp_{\mu_i} \subseteq \symgp_{\lambda_i}^{d_i}\cap \symgp_{\mu_i},\]
and thus $\symgp_{\gamma_i}\in \mathscr{Y}_i$ as $d_i\notin \symgp_{\mu_i}$. Thus if both $d_i \neq 1$, $(\symgp_{\gamma_1}, \symgp_{\gamma_2})\in \sety$.

If without loss of generality, $d_1 = 1$ and $d_2\neq 1$, then:
\[\symgp_{\gamma_1}\subseteq \symgp_{\tau_1}\cap \symgp_{\mu_1}\subseteq \symgp_{\tau_1}\subseteq \symgp_{\lambda_1} = \symgp_{\lambda_1}^1 \cap \symgp_{\mu_1}.\]
Thus as $(d_1,d_2)\notin (\symgp_{\mu_1},\symgp_{\mu_2})$, by the previous argument for $d_2$, then $(\symgp_{\gamma_1}, \symgp_{\gamma_2})\in \sety$ by definition.

Finally if both $d_1 = d_2 = 1$, then we have $(\symgp_{\gamma_1}, \symgp_{\gamma_2})\in\setx$, as each $\symgp_{\gamma_i}\subseteq \symgp_{\tau_i}$ and $(\symgp_{\tau_1}, \symgp_{\tau_2})\in \setx$. As $\symgp_{\gamma_i}\subset \symgp_{\lambda_i}$, we can conclude with Lemma~\ref{3.4II} that $(\symgp_{\gamma_1}, \symgp_{\gamma_2})\in \sety$.

Thus in all cases indecomposable summands of $L$ are relatively $\sety$-projective as $(\hecke_{\mu_1}, \hecke_{\mu_2})$-bimodules, and hence so is $M$.
\end{proof}
We can now fully state our Green correspondence for bimodules of Hecke algebras, generalising~\cite[Theorem 3.6]{du}.
\begin{theorem}[Green correspondence for bimodules]
\label{green}
We have the following correspondence:
\begin{enumerate}[(a)]
\item Let $M$ be an indecomposable $(\hecke_{\sigma_1},\hecke_{\sigma_2})$-bimodule with vertex $(\symgp_{\tau_1},\symgp_{\tau_2})\in \setz$. Then there is a unique indecomposable summand $f(M)$ of $M$ as an $(\hecke_{\mu_1}, \hecke_{\mu_2})$-bimodule, with vertex $(\symgp_{\tau_1},\symgp_{\tau_2})$, and
\[M \cong f(M) \oplus Y\]
as $(\hecke_{\mu_1}, \hecke_{\mu_2})$-bimodules, where each indecomposable summand of $Y$ has a vertex in $\sety$.
\item Let $N$ be an indecomposable $(\hecke_{\mu_1},\hecke_{\mu_2})$-bimodule with vertex $(\symgp_{\tau_1},\symgp_{\tau_2})\in\setz$. Then there is a unique indecomposable summand $g(N)$ of $\hecku_{\sigma_1, \sigma_2}\otimes_{\hecku_{\mu_1,\mu_2}} N$, with vertex $(\symgp_{\tau_1},\symgp_{\tau_2})$ and
\[\hecku\otimes_{\hecku_{\mu_1,\mu_2}}N \cong g(N) \oplus X\]
where each indecomposable summand of $X$ has a vertex in $\setx$.
\item Furthermore for $M$ and $N$ as described above, $f(g(N))\cong N$, and $g(f(M))\cong M$. 
\end{enumerate}
Hence this gives a one-to-one correspondence between $(\hecke_{\sigma_1},\hecke_{\sigma_2})$-bimodules, and $(\hecke_{\mu_1}, \hecke_{\mu_2})$-bimodules which have vertices in $\setz$.
\end{theorem}
The proof of this is largely identical to that of~\cite[Theorem 3.6]{du}. Although we are working with more general $\lambda_i, \mu_i$, and with $\sigma_i$ instead of $(n)$, the proof follows through in the same way as we still have the key relationship (\ref{key_green}) between our subgroups, and our Lemma~\ref{3.4II} and Corollary~\ref{3.5} take the place of~\cite[Lemma 3.4, Corollary 3.5]{du}. Thus the double sum in the Mackey formula is fully accounted for. Although this correspondence will hold for any $(\symgp_{\tau_1}, \symgp_{\tau_2})\in \setz$, we will typically use it in the simpler case when $\tau_i = \lambda_i$.

We can strengthen our Green correspondence, as the ideas of~\cite[\S 12]{alperin} happily carry over to bimodules of Hecke algebras, affording us the following analogue of~\cite[Theorem 12.2]{alperin}:
\begin{theorem}
\label{maps}
Let $M$ be an indecomposable $(\hecke_{\sigma_1}, \hecke_{\sigma_2})$-bimodule with vertex $(\symgp_{\lambda_1},\symgp_{\lambda_2})$, and indecomposable $(\hecke_{\mu_1}, \hecke_{\mu_2})$-bimodule $f(M)$ its Green correspondent. If $U$ is an indecomposable $(\hecke_{\sigma_1}, \hecke_{\sigma_2})$-bimodule and $f(M)\mid U$ as $(\hecke_{\mu_1}, \hecke_{\mu_2})$-bimodules, then $M\cong U$.
\end{theorem}
We can also form the following corollary which will be useful in later sections.
\begin{corollary}
\label{maps_corol}
Let $M$ and $f(M)$ be as in Theorem~\ref{maps}. If $U$ is a $(\hecke_{\sigma_1}, \hecke_{\sigma_2})$-bimodule, then $M\mid U$ as $(\hecke_{\sigma_1}, \hecke_{\sigma_2})$-bimodules, if and only if $f(M)\mid U$ as $(\hecke_{\mu_1}, \hecke_{\mu_2})$-bimodules.
\end{corollary}
\begin{proof}
Take $U = U_1\oplus \dots \oplus U_t$ a decomposition of $U$ into direct summands as $(\hecke_{\sigma_1}, \hecke_{\sigma_2})$-bimodules. As $M$ is indecomposable, then $M\mid U$ means that $M = U_i$ some $1\leq i \leq t$. Hence we get $f(M)\mid U_i \mid U$ as $(\hecke_{\mu_1},\hecke_{\mu_2})$-bimodules. For the other direction if $f(M)\mid U$ as $(\hecke_{\mu_1},\hecke_{\mu_2})$-bimodules, then $f(M)\mid U_i$ for some $i$, hence by Theorem~\ref{maps}, $M\cong U_i$ and so $M\mid U$.
\end{proof}

%% file: section4.tex
\section{A Brauer correspondence for Hecke algebras}
Now we have a version of the Green correspondence, the next logical step is to form a type of Brauer correspondence for blocks of Hecke algebras, giving results akin to Brauer's first main theorem (see for example~\cite[Theorem 14.2]{alperin}). To begin this process, we start with the following definition, an analogue of the one given for finite groups in~\cite[\S 14]{alperin}.
\begin{definition}
For $\mu \vDash n$, let $b$ be a block of $\hecke_\mu$, and $B$ a block of $\hecke_n$. We say $B$ is the \emph{Brauer correspondent} of $b$, and write $b^{\hecke_n} = B$, if $b\mid B$ as $(\hecke_\mu, \hecke_\mu)$-bimodules, and $B$ is the unique block of $\hecke_n$ with this property.
\end{definition}
As $\hecke_\mu \mid \hecke_n$ as $(\hecke_\mu, \hecke_\mu)$-bimodules (consider the decomposition given by $(\symgp_\mu,\symgp_\mu)$-double coset representatives), $b$ will always occur in the restriction of at least one block, but there is no prior guarantee that its Brauer correspondent will exist, as $b$ may occur in the restriction of more than one block. We first state some general properties of Brauer correspondents, omitting the proofs as they are largely identical to those in~\cite[Lemma 14.1]{alperin}.
\begin{lemma} 
\label{BCL1}
Let $b$ be a block of $\hecke_\mu$ for $\mu \vDash n$ with vertex $(\symgp_{\tau_1},\symgp_{\tau_2})$ as a $(\hecke_\mu,\hecke_\mu)$-bimodule. Then if $b^{\hecke_n}$ is defined, $(\symgp_{\tau_1},\symgp_{\tau_2})$ is contained in a vertex of $b^{\hecke_n}$.
\end{lemma}
\begin{lemma}
\label{BCL2}
Let $\symgp_\lambda\subseteq \symgp_\mu \subseteq \symgp_n$ be a chain of parabolic subgroups of $\symgp_n$. If $b$ is a block of $\hecke_\lambda$, and all three of $b^{\hecke_n}$, $b^{\hecke_\mu}$ and $(b^{\hecke_\mu})^{\hecke_n}$ are defined, then $(b^{\hecke_\mu})^{\hecke_n} = b^{\hecke_n}$.
\end{lemma}

\subsection{Existence of Brauer correspondents}

Let $a\geq 0$, $d \geq 1$ and $n = a + de$. Define compositions of $n$: 
\begin{align*}
\mu &= (a,de), \\ \alpha &= (a, 1^{de}), \\ \tau &= (1^a, de), 
\end{align*}
so $\symgp_\alpha\times\symgp_\tau = \symgp_\mu$. Recall, from Lemma~\ref{mu-mu_cosets}, we have the following description of $\mathscr{D}_{\mu, \mu}^{(n)}$:
\[\mathscr{D}^{(n)}_{\mu,\mu}=\left\{d_k = \prod_{i=1}^k (a-k+i,a+i) \,\bigg|\, k=0,\dots,\text{min}(a,de)\right\}.\]
This description tells us that for $i = 0, \dots, \text{min}(a,de)$, we have that $\symgp_{\nu_i} := \symgp_{\mu}^{d_i} \cap \symgp_\mu$ has corresponding composition $\nu_i = (a-i, i, i, de - i)$. We define compositions: 
\begin{align*}
\tau_i &= (1^{a+i}, de-i),\\
\tau_i' &= (1^a, i, 1^{de-i}),\\
\tilde{\tau_i} &= (1^a, i, de-i),\\
\alpha_i &= (1^{a-i}, i, 1^{de}), \\
\tilde{\alpha_i} &= (a-i,i,1^{de}).
\end{align*}
Note in particular that $\symgp_{\tilde{\tau_i}} = \symgp_{\tau_i}\times \symgp_{\tau_i'}$ and $\symgp_{\nu_i} = \symgp_{\tilde{\alpha_i}}\times \symgp_{\tilde{\tau_i}}$. This lets us present the following technical lemma:
\begin{lemma}
\label{main_h_i}
For $0\leq i \leq \min(a,de)$, as an $(\hecke_\tau,\hecke_\tau)$-module, $\hecke_\mu T_{d_i} \hecke_\mu$ is $(\symgp_{\tau_i},\symgp_\tau)$-projective.
\end{lemma}
\begin{proof}
By \cite[Proposition 4.4]{mathas}, every element $w\in\symgp_n$ can be uniquely represented as a product $w = gd_i h$ for $g\in \symgp_\mu$, $d_i \in \mathscr{D}^{(n)}_{\mu,\mu}$ and $h\in \mathscr{R}^\mu_{\nu_i}$, with $\ell(w) = \ell(g) + \ell(d_i) + \ell(h)$. Hence the following gives us an $F$-basis for $\hecke_\mu T_{d_i} \hecke_\mu$:
\[\{ T_{g d_i h} = T_g T_{d_i} T_h : g \in \symgp_\mu, h\in \mathscr{R}^\mu_{\nu_i}\}.\]
Furthermore, as $\symgp_\mu = \symgp_\alpha \times \symgp_\tau$, we can further categorise our basis (as for each $g\in \symgp_\mu$, there exists unique $x\in \symgp_\alpha$ and $y\in \symgp_\tau$ with $g = xy$, and in addition $\ell(g) = \ell(x) + \ell(y)$). In the same vein, $\mathscr{R}^\mu_{\nu_i} = \mathscr{R}^\alpha_{\tilde{\alpha_i}} \times \mathscr{R}^\tau_{\tilde{\tau_i}}$, so $h\in \mathscr{R}^\mu_{\nu_i}$ can be written uniquely as $h = h_1 h_2$ with $h_1 \in \mathscr{R}^\alpha_{\tilde{\alpha_i}}$, $h_2\in \mathscr{R}^\tau_{\tilde{\tau_i}}$, and $\ell(h_1) + \ell(h_2) = \ell(h)$. In particular, $T_{h_1}T_{h_2} = T_{h_2}T_{h_1}$ as $h_1$ commutes with $\symgp_\tau$. 
Also, as $\symgp_{\alpha_i}\subseteq \symgp_\alpha$, we can write $x\in \symgp_\alpha$ uniquely as $x = x_1x_2$ with $x_1\in \mathscr{L}^{\alpha}_{\alpha_i}$, $x_2\in \symgp_{\alpha_i}$ and $\ell(x_1) + \ell(x_2) = \ell(x)$. Therefore our $F$-basis for $\hecke_\mu T_{d_i} \hecke_\mu$ can be written as:
\[\{ T_{x_1} T_{x_2} T_y T_{d_i} T_{h_1}T_{h_2} : x_1\in \mathscr{L}^{\alpha}_{\alpha_i}, x_2\in \symgp_{\alpha_i}, y\in \symgp_\tau, h_1 \in \mathscr{R}^\alpha_{\tilde{\alpha_i}}, h_2\in \mathscr{R}^\tau_{\tilde{\tau_i}}\}.\]
Now for some fixed $x_1\in \mathscr{L}^\alpha_{\alpha_i}$ and $h_1\in \mathscr{R}^\alpha_{\tilde{\alpha_i}}$, consider the vector subspace 
\[M_{x_1, h_1}:=\langle T_{x_1} T_{x_2} T_y T_{d_i} T_{h_1}T_{h_2} : x_2\in \symgp_{\alpha_i}, y\in \symgp_\tau, h_2\in \mathscr{R}^\tau_{\tilde{\tau_i}}\rangle.\]
We show that this is closed under left and right multiplication by elements of $\hecke_\tau$, so is a $(\hecke_\tau, \hecke_\tau)$-bimodule.

Let $s_j = (j, j+1)\in \symgp_\tau$. Multiplying basis element $m_{x_2,y,h_2} := T_{x_1}T_{x_2}T_{y}T_{d_i}T_{h_1}T_{h_2}$ by $T_{j}$ on the left:
\begin{align*} T_jm_{x_2,y,h_2} &= T_{j}T_{x_1}T_{x_2}T_y T_{d_i}T_{h_1}T_{h_2} \\ &= T_{x_1}T_{x_2}T_{j}T_y T_{d_i} T_{h_1}T_{h_2} \\
& = \begin{cases} T_{x_1}T_{x_2}T_{s_j y} T_{d_i}T_{h_1}T_{h_2} & \text{ if }\ell(s_j y) > \ell(y), \\
			(q-1)T_{x_1}T_{x_2}T_{y} T_{d_i}T_{h_1}T_{h_2} + qT_{x_1}T_{x_2}T_{s_jy}T_{d_i}T_{h_1}T_{h_2} & \text{ if } \ell(s_j y) < \ell(y),
			\end{cases}
\end{align*}
as $\hecke_\alpha$ and $\hecke_\tau$ commute. As $y, s_j$ and hence $s_jy\in \symgp_\tau$, $M_{x_1, h_1}$ is a left $\hecke_\tau$-module. We now need to check right multiplication.
\begin{align*} m_{x_2,y,h_2}T_j &= T_{x_1}T_{x_2}T_y T_{d_i}T_{h_1}T_{h_2}T_{j}  \\ &= T_{x_1}T_{x_2}T_y T_{d_i} T_{h_2}T_{j}T_{h_1} \\
& = \begin{cases} T_{x_1}T_{x_2}T_{y} T_{d_i}T_{h_2s_j}T_{h_1} & \text{ if }\ell(h_2 s_j) > \ell(h_2), \\
			(q-1)T_{x_1}T_{x_2}T_y T_{d_i}T_{h_1}T_{h_2} + qT_{x_1}T_{x_2}T_{y}T_{d_i}T_{h_2s_j}T_{h_1} & \text{ if } \ell(h_2 s_j) < \ell(h_2).
			\end{cases}
\end{align*}
Hence it is sufficient to show that $T_{x_1}T_{x_2}T_{y}T_{d_i}T_{h_2s_j}T_{h_1}\in M_{x_1, h_1}$. To do this we split into cases dependent on whether $j$ and $j+1$ are in the same row of $(\mathfrak{t}^{\nu_i})\cdot h_2$ or not. For this, we will liberally use~\cite[Proposition 3.3, Corollary 4.4]{mathas}.
\begin{itemize}
\item If they are not in the same row, then $(\mathfrak{t}^{\nu_i})\cdot h_2s_j$ is a row-standard tableau. Hence $h_2s_j\in \mathscr{R}^{\tau}_{\tilde{\tau_i}}$, thus our element lies in $M_{x_1, h_1}$.
\item If they are in the same row, then there exists $k$ such that $h_2(k) = j$ and $h_2(k+1) = j+1$, as $(\mathfrak{t}^{\nu_i})\cdot h_2$ is row standard, hence $j$ must be next to $j+1$. Therefore there exists an elementary transposition $s_k\in \symgp_\tau$ with $s_k h_2 = h_2 s_j$. Furthermore, as $h_2$ is a minimal right coset representative, $\ell(s_k h_2) = \ell(s_k) + \ell(h_2)$, thus $T_{s_k h_2} = T_{s_k} T_{h_2}$. Therefore:
\[T_{x_1}T_{x_2}T_y T_{d_i}T_{h_2s_j}T_{h_1} = T_{x_1}T_{x_2}T_y T_{d_i}T_{s_k}T_{h_2}T_{h_1}.\]
In addition as $k$ and $k+1$ are in the same row of $\mathfrak{t}^{\tilde{\tau_i}}$,  $s_k\in \symgp_{\tau_i}$ or $s_k\in \symgp_{\tau_i'}$. We further split based on these cases.
\begin{itemize}
\item If $s_k\in \symgp_{\tau_i}$, then:
\[ T_{x_1}T_{x_2}T_y T_{d_i}T_{s_k}T_{h_1}T_{h_2} =  T_{x_1}T_{x_2}T_y T_{s_k}T_{d_i}T_{h_1}T_{h_2},\]
which lies in $M_{x_1, h_1}$ as before.
\item If $s_k\in \symgp_{\tau_i'}$, then:
\[ T_{x_1}T_{x_2}T_y T_{d_i}T_{s_k}T_{h_1}T_{h_2} = T_{x_1}T_{x_2}T_y T_{s_{k-i}}T_{d_i}T_{h_1}T_{h_2}.\]
Note $s_{k-i}d_i = d_i s_k$ as $d_i$ is both a left and right coset representative of $\symgp_\mu$ in $\symgp_n$. Furthermore, by minimality $\ell(s_{k-i}d_i) = 1 + \ell(d_i) = \ell(d_i s_k)$. Then:
\[T_{x_1}T_{x_2}T_y T_{s_{k-i}}T_{d_i}T_{h_1}T_{h_2} = T_{x_1}(T_{x_2} T_{s_{k-i}})T_yT_{d_i}T_{h_1}T_{h_2},\]
and as $x_2, s_{k-i} \in \symgp_{\alpha_i}$, we have an element in $M_{x_1, h_1}$.
\end{itemize}
So if $j$ and $j+1$ are in the same row, our element again lies in $M_{x_1, h_1}$.
\end{itemize}
Thus $M_{x_1, h_1}$ is closed under right multiplication by $\hecke_\tau$. As multiplication in $\hecke_n$ is associative, for a fixed $x_1$ and $h_1$, $M_{x_1, h_1}$ is an $(\hecke_\tau, \hecke_\tau)$-submodule of $\hecke_\mu T_{d_i} \hecke_\mu$. Furthermore, since we have described bases for the bimodules involved as vector spaces, we have a direct sum decomposition of $\hecke_\mu T_{d_i} \hecke_\mu$ as a $(\hecke_\tau, \hecke_\tau)$-bimodule:
\[\hecke_\mu T_{d_i} \hecke_\mu = \bigoplus_{x_1 \in \mathscr{L}^\alpha_{\alpha_i}, h_1\in \mathscr{R}^\alpha_{\tilde{\alpha_i}}} M_{x_1, h_1} \cong \bigoplus_{x_1 \in \mathscr{L}^\alpha_{\alpha_i}, h_1\in \mathscr{R}^\alpha_{\tilde{\alpha_i}}} M_{1,1},\]
as our above calculations show that the $T_{x_1}$ and $T_{h_1}$ have no effect on left or right multiplication by $\hecke_\tau$. Therefore for our purposes, it suffices to show that $M_{1,1}$ is $(\symgp_{\tau_i}, \symgp_{\tau})$-projective as an $(\hecke_\tau, \hecke_\tau)$-module.

To do this, consider the vector space $N:= \langle T_x T_y T_{d_i} T_h : x\in \symgp_{\alpha_i}, y\in \symgp{\tau_i}, h\in \mathscr{R}^{\tau}_{\tilde{\tau_i}}\rangle$.
This is a $(\hecke_{\tau_i}, \hecke_\tau)$-bimodule, by similar calculations to those above. Looking at the bases of $N$ and $M_{1,1}$, we can see that:
\[M_{1,1}\cong \hecke_{\tau}\otimes_{\hecke_{\tau_i}} N\]
as $(\hecke_\tau, \hecke_\tau)$-bimodules. Thus $M_{1,1}$ is $(\symgp_{\tau_i}, \symgp_\tau)$-projective as a $(\hecke_\tau, \hecke_\tau)$-bimodule, hence the same holds for $\hecke_\mu T_{d_i} \hecke_\mu$.
\end{proof}
\begin{corollary}
\label{useful_corl}
In situation of Lemma~\ref{main_h_i}, let $M$ be a direct summand of $\hecke_\mu T_{d_i} \hecke_\mu$ as a $(\hecke_\mu, \hecke_\mu)$-bimodule. Then if $M$ is $(\symgp_\tau, \symgp_\tau)$-projective, it is also $(\symgp_{\tau_i}, \symgp_\tau)$-projective.
\end{corollary}
\begin{proof} This follows from Lemma~\ref{main_h_i} and the bimodule analogue of Corollary~\ref{transitivity}.
\end{proof}
We now introduce the following type of parabolic subgroup.
\begin{definition}
A parabolic subgroup $\symgp_\lambda\subseteq\symgp_n$ is \emph{fixed-point-free} if the corresponding composition $\lambda = (\lambda_1, \dots, \lambda_s) \vDash n$ has $\lambda_i > 1$ for all $1 \leq i \leq s$. 

Suppose that we have a composition $\gamma = (1^a, b)\vDash n$. We say a parabolic subgroup $\symgp_\lambda$ is a fixed-point-free subgroup of $\symgp_\gamma$ if the corresponding composition $\lambda = (1^a, \lambda_1, \dots, \lambda_s)$ has $\lambda_i > 1$ for all $1 \leq i \leq s$.
\end{definition}
This corresponds to the notion that no element of $\{1, \dots, n\}$ (or $\{a + 1, \dots, n\}$ in the second case) is fixed by all elements of $\symgp_\lambda$. 
Note that for any fixed-point-free subgroup $\symgp_\lambda$ of $\symgp_\tau$, we have $N_{\symgp_n}(\symgp_\lambda) \subseteq \symgp_\mu$. Hence by Theorem~\ref{green}, in this case, we have a bijection between $(\hecke_n, \hecke_n)$ and $(\hecke_\mu,\hecke_\mu)$-bimodules with vertex $(\symgp_\lambda, \symgp_\lambda)$.
\begin{theorem} 
\label{double_coset_vertices} Let $1\neq d_i\in \mathscr{D}_{\mu\mu}$, and $\symgp_\lambda$ a proper fixed-point-free parabolic subgroup of $\symgp_\tau$. Then no indecomposable summand of $\hecke_\mu T_{d_i}\hecke_\mu$ as an $(\hecke_\mu, \hecke_\mu)$-bimodule has vertex $(\symgp_\lambda, \symgp_\lambda)$.
\end{theorem}
\begin{proof}
Consider $M$ a direct summand of $\hecke_\mu T_{d_i} \hecke_\mu$ as a $(\hecke_\mu, \hecke_\mu)$-bimodule. If $M$ has vertex $(\symgp_\lambda, \symgp_\lambda)$, then as $\symgp_\lambda \subset \symgp_\tau$, we get that 
$M$ is $(\symgp_\tau, \symgp_\tau)$-projective by transitivity of induction. Corollary~\ref{useful_corl} tells us that $M$ is $(\symgp_{\tau_i}, \symgp_\tau)$-projective as a $(\hecke_\mu, \hecke_\mu)$-bimodule. However, $M$ has vertex $(\symgp_\lambda, \symgp_\lambda)$ which means that some conjugate of $\symgp_\lambda$ is contained in $\symgp_{\tau_i}$.  As $\symgp_{\lambda}$ is fixed-point-free in $\symgp_\tau$, it contains an element of cycle type $\lambda_1\dots\lambda_s$. No elements in $\symgp_{\tau_i}$ can have this cycle type as there are not enough indices, hence $M$ cannot have vertex $(\symgp_\lambda, \symgp_\lambda)$.
\end{proof}
\begin{corollary}
\label{brauer_existence}
Let $b$ be a block of $\hecke_\mu$ with vertex $(\symgp_\lambda, \symgp_\lambda)$, where $\symgp_\lambda$ is a fixed-point-free parabolic subgroup of $\symgp_\tau$. Then $b^{\hecke_n}$ exists. 
\end{corollary}
\begin{proof} Decomposing $\hecke_n$ as a $(\hecke_\mu, \hecke_\mu)$-bimodule using double cosets gives
\[\hecke_n = \hecke_\mu \oplus \bigoplus_{1\neq d\in \mathscr{D}^{(n)}_{\mu\mu}} \hecke_\mu T_d\hecke_\mu.\]
Now $b$ occurs once as a summand of $\hecke_\mu$, and does not appear as a direct summand of any $\hecke_\mu T_d \hecke_\mu$ for $d\neq 1$ by Theorem~\ref{double_coset_vertices}, as no indecomposable summands of this have the required vertex. Therefore $b$ occurs exactly once in this direct sum decomposition, so there must be a unique block of $\hecke_n$ which restricts to contain $b$.
\end{proof}
This finally lets us state our Brauer correspondence. Note that this is not as general as the Brauer correspondence stated in~\cite[Theorem 14.2]{alperin}, as we require $\symgp_\mu$ to have two parts, and need $\symgp_\lambda$ to be a fixed-point-free subgroup of $\symgp_\tau$. This is instead of only requiring $N_{\symgp_n}(\symgp_\lambda)\subseteq \symgp_\mu$ in the classical Brauer correspondence. Nevertheless, as we will show in the following sections, all blocks have vertices satisfying this condition, and thus it will give a complete characterisation of the vertices for the blocks of $\hecke_n$. 
\begin{theorem}[Brauer correspondence for Hecke algebras]
\label{new_brauer}
Let $n = a + de$, with $\mu = (a, de)$, $\tau = (1^a, de)$ and $\symgp_\lambda$ a fixed-point-free parabolic subgroup of $\symgp_\tau$. Then there is a one-to-one correspondence between blocks of $\hecke_\mu$ with vertex $(\symgp_\lambda,\symgp_\lambda)$ and blocks of $\hecke_n$ with the same vertex. 
\end{theorem}
\begin{proof}
First let $b$ be a block of $\hecke_\mu$ with vertex $(\symgp_\lambda, \symgp_\lambda)$. Then $b^{\hecke_n}$ exists by Corollary~\ref{brauer_existence}. As $\symgp_\lambda\subseteq \symgp_\tau$ is fixed-point-free, we have $N_{\symgp_n}(\symgp_\lambda)\subseteq \symgp_\mu$. Hence we can use Theorem~\ref{green} to show that $b$ has a Green correspondent, and by Theorem~\ref{maps}, this Green correspondent must be $b^{\hecke_n}$. This correspondence gives us that $b^{\hecke_n}$ has the same vertex as $b$, and as the Green correspondence is a bijection, in particular the map $b\mapsto b^{\hecke_n}$ must be injective.

Now let $B$ be a block of $\hecke_{n}$ with vertex $(\symgp_\lambda, \symgp_\lambda)$. By Lemma~\ref{3.2i}, there is an indecomposable $(\hecke_\mu, \hecke_\mu)$-bimodule $N$ with vertex $(\symgp_\lambda, \symgp_\lambda)$, and $N \mid B$ as $(\hecke_\mu, \hecke_\mu)$-bimodules. Theorem~\ref{double_coset_vertices} tells us that $N$ must be a direct summand of $\hecke_\mu$ and hence is a block of $\hecke_\mu$. Therefore by Corollary~\ref{brauer_existence}, $N^{\hecke_n}$ exists, and by the first part of this proof, $N^{\hecke_n}=B$. This shows us that the map $b\mapsto b^{\hecke_n}$ is surjective onto blocks with vertex $(\symgp_\lambda, \symgp_\lambda)$, and hence defines the required one-to-one correspondence.
\end{proof}

In particular note that Brauer corresponding blocks are also Green correspondents in the sense of Theorem~\ref{green}.

\subsection{Finding Brauer correspondents}

Now we know they exist, we want to be able to identify the Brauer correspondent of a given block. We begin by proving a theorem which links Brauer correspondents to Green correspondents, similar to~\cite[Corollary 14.4]{alperin}. Throughout this section, we will denote the central idempotent of the block $b$ by $e_b$, and that of $B$ by $e_B$.
\begin{theorem}
Let $\mu\vDash n$ and $b$ a block of $\hecke_\mu$ whose Brauer correspondent $B = b^{\hecke_n}$ exists. Let $\lambda\vDash n$ with $N_{\symgp_n}(\symgp_\lambda)\subseteq\symgp_\mu$, and suppose $N$ is an indecomposable $\hecke_\mu$-module lying in $b$, with vertex $\symgp_\lambda$. Then $g(N)$, the Green correspondent of $N$, lies in $B$.
\end{theorem}
\begin{proof}
Note first that the Green correspondent of $N$ exists by~\cite[Theorem 3.6]{du}. Thus
\[N \otimes_{\hecke_\mu} \hecke_n \cong g(N) \oplus Q\]
where $g(N)$ is indecomposable, has vertex $\symgp_\lambda$, and the indecomposable summands of $Q$ all have vertices that are strictly smaller than $\symgp_\lambda$. Suppose that $g(N)e_B = 0$. Then: 
\[N\otimes_{\hecke_\mu} B = (N\otimes_{\hecke_\mu}\hecke_n)e_B \cong g(N)e_B \oplus Qe_B = Qe_B\]
and hence each indecomposable summand of $N\otimes_{\hecke_\mu} B$ has vertex strictly smaller than $\symgp_\lambda$ as an $\hecke_n$-module. Thus its restriction down to $\hecke_\mu$ must also have vertices strictly smaller than $\symgp_\lambda$, by the Mackey formula. By the definition of Brauer correspondents, $B \cong b \oplus P$ as $(\hecke_\mu, \hecke_\mu)$-bimodules, for some $(\hecke_\mu, \hecke_\mu)$-bimodule $P$. Thus as $\hecke_\mu$-modules:
\[N \otimes_{\hecke_\mu} b \mid N\otimes_{\hecke_\mu} B.\]

However,
\[N\otimes_{\hecke_\mu} b = N \otimes_{\hecke_\mu} e_b \hecke_\mu = N e_b\otimes_{\hecke_\mu}\hecke_\mu\cong N e_b = N\]
since $N$ lies in the block $b$, so $N \mid N\otimes_{\hecke_\mu} B$. This is a contradiction, as the indecomposable summands of $N\otimes_{\hecke_\mu} B$ have vertices strictly smaller than $\symgp_\lambda$ as an $\hecke_\mu$-module. Hence $Me_B = M$ and so $M$ lies in the block $B$ of $\hecke_n$.
\end{proof}
Thus searching for Green correspondents of modules in our block $b$ gives a way to identify $b^{\hecke_n}$. We summarise this test in the following corollary.
\begin{corollary} 
\label{block_test}
Let $\mu = (a, de), \tau = (1^a, de), \gamma \vDash n$, and $\symgp_\lambda$ a fixed-point-free parabolic subgroup of $\symgp_\tau$. Suppose $\symgp_\gamma \subseteq N_{\symgp_n}(\symgp_\gamma)\subseteq \symgp_\mu$, and let $b$ be a block of $\hecke_\mu$ with vertex $(\symgp_\lambda, \symgp_\lambda)$. If $N$ is an indecomposable $\hecke_\mu$-module in $b$ with vertex $\symgp_\gamma$, and its Green correspondent $g(N)$ in $\hecke_n$ lies in $B$, then $B = b^{\hecke_n}$. 
\end{corollary}
\begin{proof} Theorem~\ref{new_brauer} guarantees that $b^{\hecke_n}$ exists, and by the preceding theorem, $g(N)$ lies in $b^{\hecke_n}$. 
\end{proof}
Before concluding this section, we present one last theorem to aid us when computing the vertex of a particular block; in effect this gives a lower bound on the possible vertex.
\begin{theorem} 
\label{blocks_and_modules} Let $B$ be a block of $\hecke_n$ with vertex $(\symgp_{\lambda_1},\symgp_{\lambda_2})$, and $M$ an indecomposable right $\hecke_n$-module that lies in $B$ with vertex $\symgp_\gamma$. Then there exists $x\in \symgp_n$ with $\symgp_\gamma^x\subseteq \symgp_{\lambda_2}$.
\end{theorem}
\begin{proof}
As $M$ has vertex $\symgp_\gamma$, there exists some $\hecke_\gamma$-module $N$ with $M \mid N \otimes_{\hecke_{\gamma}} \hecke_n$ by Lemma~\ref{3.2ii}. Multiplying both sides by $e_B$:
\[M = Me_B \mid N\otimes_{\hecke_{\gamma}} \hecke_n e_B = N\otimes_{\hecke_\gamma} B.\]
By Higman's criterion, there exists some $(\hecke_{\lambda_1},\hecke_{\lambda_2})$-bimodule $Q$ with $B\mid \hecke_n\otimes_{\hecke_{\lambda_1}} Q\otimes_{\hecke_{\lambda_2}} \hecke_n$ as $(\hecke_n, \hecke_n)$-bimodules. By restricting both sides, the same holds true as $(\hecke_\gamma, \hecke_n)$-bimodules. Combining this with our previous statement, as $N$ is a $\hecke_\gamma$-module, means that as $\hecke_n$-modules:
\[M \mid N\otimes_{\hecke_\gamma} (\hecke_n\otimes_{\hecke_{\lambda_1}} Q\otimes_{\hecke_{\lambda_2}}\hecke_n) \cong (N\otimes_{\hecke_\gamma} (\hecke_n\otimes_{\hecke_{\lambda_1}} Q))\otimes_{\hecke_{\lambda_2}}\hecke_n,\]
by associativity. Setting $V = N\otimes_{\hecke_\gamma}(\hecke_n\otimes_{\hecke_{\lambda_1}} Q)$, which is an $\hecke_{\lambda_2}$-module, we get
\[M \mid V \otimes_{\hecke_{\lambda_2}} \hecke_n,\]
hence $M$ is relatively $\symgp_{\lambda_2}$-projective, and thus some conjugate of $\symgp_\gamma$ lies inside $\symgp_{\lambda_2}$.
\end{proof}

%% file: section5.tex
\section{Specht modules}

The main aim of this section is to understand enough about Specht modules for $\hecke_n$ to use them when applying Corollary~\ref{block_test}. We recall Specht modules for Hecke algebras as in~\cite[\S3]{mathas}. Note that these Specht modules correspond to the dual of the Specht modules used by Dipper and James in~\cite{djblocks}.

Recall from~\cite[\S 3]{mathas} the following definitions and notation. For $\lambda \vdash n$, let $m_\lambda = \sum_{w\in \symgp_\lambda} T_w$. For $\mathfrak{s}, \mathfrak{t}\in \std{\lambda}$, denote $m_{\mathfrak{s}\mathfrak{t}} = T_{d(\mathfrak{s})^{-1}}m_\lambda T_{d(\mathfrak{t})}$, where $d(\mathfrak{s})$ is the minimal right coset representative sending the standard tableau $\mathfrak{t}^\lambda$ to $\mathfrak{s}$. By~\cite[Theorem 3.20]{mathas}, the following set is an $F$-basis for $\hecke_n$.
\[ \{m_{\mathfrak{st}} : \mathfrak{s}, \mathfrak{t} \in \std{\lambda}\text{ for some }\lambda \vdash n\} \] Let $\check{\hecke}^\lambda$ be the two-sided ideal of $\hecke_n$ with basis
\[\{ m_{\mathfrak{u}\mathfrak{v}} : \mathfrak{u}, \mathfrak{v}\in \std{\nu}\text{ for some }\nu \rhd \lambda\},\]
where $\rhd$ denotes the dominance ordering on partitions of $n$, and denote $m_\mathfrak{t} = m_{\mathfrak{t}^\lambda \mathfrak{t}} + \check{\hecke}^\lambda$, for $\mathfrak{t} \in \std{\lambda}$. Then the \emph{Specht module} $S^\lambda$ is the free $F$-module with basis 
\[\{m_\mathfrak{t}: \mathfrak{t}\in \std{\lambda}\}.\]
Rules for multiplication by elements of $\hecke_n$ in this module can be gained from taking~\cite[Corollary 3.4]{mathas} modulo $\check{\hecke}^\lambda$, and~\cite[Corollary 3.21]{mathas}. In particular, note that $S^{(n)}$ is the trivial module (all generators of $\hecke_n$ act by multiplication by $q$), and $S^{(1^{n})}$ is the sign module (all generators act as multiplication by $-1$).

As our goal is to use Specht modules to find blocks which are Brauer Correspondents, we need to know that the Specht modules we are looking at are indecomposable. 
\begin{lemma}
\label{indecompspecht}
Let $\lambda\vdash n$ be an $e$-restricted partition. Then $S^\lambda$ is an indecomposable $\hecke_n$-module.
\end{lemma}
\begin{proof} Using the cellular structure of $\hecke_n$, from~\cite[Corollary 2.6]{grahamlehrer} we get that $\Endo_{\hecke_n}(S^\lambda) \cong F$, and thus $S^\lambda$ is indecomposable.
\end{proof}
\begin{corollary}
Let $(1^n) \vdash n$, and $\tau \vdash n$ an $e$-core. Then both $S^\tau$ and $S^{(1^n)}$ are indecomposable $\hecke_n$-modules.
\end{corollary}
When $\tau \vdash n$ is an $e$-core, we can say even more.
\begin{proposition}
If $\tau \vdash n$ is an $e$-core, then $S^\tau$ is projective.
\end{proposition}
\begin{proof} If $\tau$ is an $e$-core, then it lies in a block of $e$-weight zero, which is semi-simple by~\cite[Theorem 1.2]{erdmann}.
\end{proof}
This means that if $\tau \vdash a$ is an $e$-core, then $S^\tau \otimes S^{1^m}$ will be an indecomposable $\hecke_a\otimes \hecke_m$-module with vertex contained in $\symgp_{(1^a,m)}$. As such, $S^\tau \otimes S^{(1^m)}$ is a good candidate to use when applying Corollary~\ref{block_test}. 

\subsection{Restriction of Specht modules}

\begin{definition}
Let $\tau  = (\tau_1,\dots, \tau_s)\vdash a$ for some positive integer $a$, and let $m$ be another positive integer with $n = a + m$. Define the \emph{extended partition} $\tilde{\tau} = (\tau_1,\dots, \tau_s, 1^m) \vdash n$ and say that $\mathfrak{t}\in \std{\tilde{\tau}}$ has an \emph{$m$-tail} if the integers $\{a+1, \dots, a+m\}$ lie in the last $m$ rows of $\mathfrak{t}$. Finally define \[\std{\tau, m}:= \{ \mathfrak{t}\in \std{\tilde{\tau}} : \mathfrak{t}\text{ has an }m\text{-tail}\}.\]
\end{definition}
Fix $n = a + m$. We have an obvious bijection between $\std{\tau}$ and $\std{\tau, m}$ by adding or removing the $m$-tail. We denote this by sending $\mathfrak{t}$ to $\tilde{\mathfrak{t}}$. In fact, the following lemma is easy to verify.
\begin{lemma}
\label{dom_prop}
Let $\tau\vdash a$, and $\mathfrak{t}, \mathfrak{s} \in \std{\tau}$. Then
\[\mathfrak{t}\rhd \mathfrak{s} \iff \tilde{\mathfrak{t}} \rhd \tilde{\mathfrak{s}}.\]
Furthermore, if $\mathfrak{t}\in \std{\tau, m}$, and $\mathfrak{v}\in \std{\tilde{\tau}}$, then $\mathfrak{v} \rhd \mathfrak{t}$ implies $\mathfrak{v}\in \std{\tau, m}$.
\end{lemma}
Let $\tau \vdash a$, and $\mu = (a,m) \vDash n$. Then we can find an interesting submodule of $S^{\tilde{\tau}}$ as follows.
\begin{lemma}
\label{sub1}
Let $S^{\tau, m}$ be the vector space inside $S^{\tilde{\tau}}$ spanned by basis elements $m_{\mathfrak{t}}$ where $\mathfrak{t}\in \std{\tau, m}$. Then $S^{\tau, m}$ is an $\hecke_\mu$-submodule of $S^{\tilde{\tau}}$.
\end{lemma}
\begin{proof}
We show that $S^{\tau,m}$ is closed under multiplication by $T_{i}$, for $s_i \in \symgp_\mu$. If $i$ and $i+1$ lie in the same column of $\mathfrak{t}$, then we can conclude using~\cite[Corollary 3.21]{mathas} and Lemma~\ref{dom_prop}.

Otherwise we split into cases, depending on whether or not $i$ and $i+1$ are in the same row of $\mathfrak{t}$. If they are in the same row, then $\mathfrak{s} = \mathfrak{t} s_i$ is not row standard, and hence by~\cite[Corollary 3.4]{mathas}, $m_\mathfrak{t} T_{i} = q m_\mathfrak{t}$, and hence lies in $S^{\tau, m}$. If $i$ and $i+1$ are not in the same row, then $\mathfrak{s}$ is standard, and contains an $m$-tail. Using~\cite[Corollary 3.4]{mathas} again:
\[m_{\mathfrak{t}}T_{i} = \begin{cases} m_{\mathfrak{s}} & \text{ if } \ell(d(\mathfrak{s})) > \ell(d(\mathfrak{t})), \\
								q m_{\mathfrak{s}} + (q-1)m_{\mathfrak{t}} & \text{ otherwise} \end{cases}\]
and so in both cases $m_{\mathfrak{t}}T_{i}\in S^{\tau, m}$.
\end{proof}
\begin{theorem}
\label{sub2}
As $\hecke_\mu$-modules, $S^{\tau, m} \cong S^\tau\otimes S^{(1^m)}$.
\end{theorem}
\begin{proof}
Let $\{m_{\tilde{\mathfrak{t}}}: \mathfrak{t}\in \std{\tau}\}$ be our standard basis of $S^{\tau,m}$, and $\{n_\mathfrak{t}\otimes \epsilon:\mathfrak{t}\in \std{\tau}\}$ be the basis of $S^\tau\otimes S^{(1^m)}$ gained from taking the standard basis of $S^\tau$ and tensoring with single basis element $\epsilon$ of $S^{(1^m)}$. Define a map $\phi:S^\tau \otimes S^{(1^m)}\to S^{\tau,m}$ by $\phi:n_\mathfrak{t}\otimes \epsilon \mapsto m_{\tilde{\mathfrak{t}}}$ extended linearly. To show $\phi$ is a $\hecke_\mu$-module isomorphism, it suffices to show that the map is a $\hecke_\mu$-module homomorphism, i.e.\ it suffices to show that:
\[\phi((n_\mathfrak{t}\otimes \epsilon)T_{i}) = \phi(n_\mathfrak{t}\otimes \epsilon)T_{i}\]
for all $s_i = (i,i+1)\in \symgp_\mu$.

First suppose that $i$ and $i+1$ are in the same row of $\mathfrak{t}^{\tau}$, (so necessarily $s_i\in \symgp_a$). Then $(n_\mathfrak{t}\otimes \epsilon)T_{i} =  n_\mathfrak{t}T_{i}\otimes \epsilon$. In $\hecke_a$, we have:
\[ m_{\mathfrak{t}^\tau\mathfrak{t}}T_i = m_{\tau}T_{d(\mathfrak{t})}T_i = qm_{\tau}T_{d(\mathfrak{t})} = qm_{\mathfrak{t}^\tau\mathfrak{t}} \]
by~\cite[Corollary 3.4]{mathas}. Thus when taken modulo $\check{\hecke}^\tau$ we get that $n_{\mathfrak{t}} T_i = qn_{\mathfrak{t}}$. By the same reasoning, $m_{\tilde{\mathfrak{t}}}T_{i} = qm_{\tilde{\mathfrak{t}}}$ when $i$ and $i+1$ are in the same row, and thus
\[\phi((n_\mathfrak{t}\otimes \epsilon)T_{i}) = \phi(n_\mathfrak{t}\otimes \epsilon)T_{i}.\]

Now suppose $i$ and $i+1$ are not in the same column, and are not in the same row (again  we must have $(i,i+1)\in \symgp_a$). Using~\cite[Corollary 3.4]{mathas}, we get that $\phi((n_\mathfrak{t}\otimes \epsilon)T_{i}) = \phi(n_\mathfrak{t}\otimes \epsilon)T_{i}$, since $\mathfrak{s} = \mathfrak{t}s_i$ is standard, and $\tilde{\mathfrak{s}} = \tilde{\mathfrak{t}}s_i$. 

It remains to deal with the case where $i$ and $i+1$ are in the same column, and we split into further cases based on whether $(i,i+1)\in \symgp_a$ or $(i,i+1)\in \symgp_m$. 
\begin{itemize}
\item First suppose that $(i,i+1)\in \symgp_a$. Note that as elements of $\hecke_n$, we have that $m_\tau = m_{\tilde{\tau}}$ and $T_{d(\mathfrak{t})} = T_{d(\tilde{\mathfrak{t}})}$. So using~\cite[Proposition 3.21]{mathas}, we have that in $\hecke_a$:
\[m_{\mathfrak{t}^\tau\mathfrak{t}}T_{i} \equiv -m_{\mathfrak{t}^\tau\mathfrak{t}} + \sum_{\mathfrak{v}\rhd \mathfrak{t}}r_{\mathfrak{v}}m_{\mathfrak{t}^\tau\mathfrak{v}} \mod{\check{\hecke}^\tau}.\]
Now if $m_{\mathfrak{u}\mathfrak{w}}$ is a basis element of $\check{\hecke}^\tau$, for $\mathfrak{u},\mathfrak{w}\in \std{\lambda}$ for some $\lambda \rhd \tau$, then
\[m_{\mathfrak{u}\mathfrak{w}} = m_{\tilde{\mathfrak{u}}\tilde{\mathfrak{w}}}\]
is a basis element of $\check{\hecke}^{\tilde{\tau}}$ as we know that $\mu\rhd \tau \implies \tilde{\mu}\rhd\tilde{\tau}$. Similarly the fact that $\mathfrak{v}\rhd \mathfrak{t}\iff \tilde{\mathfrak{v}}\rhd \tilde{\mathfrak{t}}$ gives us that:
\[m_{\mathfrak{t}^{\tilde{\tau}}\tilde{\mathfrak{t}}}T_{i} \equiv -m_{\mathfrak{t}^{\tilde{\tau}}\tilde{\mathfrak{t}}} + \sum_{\tilde{\mathfrak{v}}\rhd \tilde{\mathfrak{t}}}r_{\mathfrak{v}}m_{\mathfrak{t}^{\tilde{\tau}}\tilde{\mathfrak{v}}} \mod{\check{\hecke}^{\tilde{\tau}}}.\]
Thus again multiplication is the same in both modules.
\item Finally when $i$ and $i+1$ both lie in the $m$-tail (so $T_{d(\mathfrak{t})}$ commutes with $T_{i}$):
\[\phi((n_\mathfrak{t}\otimes \epsilon)T_{i}) = \phi(n_\mathfrak{t}\otimes(\epsilon T_{i})) = \phi(-n_\mathfrak{t}\otimes \epsilon) = -m_{\tilde{\mathfrak{t}}}.\]
So it suffices to show that $m_{\tilde{\mathfrak{t}}}T_{i} = -m_{\tilde{\mathfrak{t}}}$, i.e.\ $m_{\tilde{\tau}}T_{d(\mathfrak{\tilde{t}})}(1+T_{i})\in \check{\hecke}^{\tilde{\tau}}$. Writing $\tilde{\tau} = (\tau_1, \dots, \tau_s, 1^l, 1^m)$ where each $\tau_i > 1$, we have that:
\begin{align*}
m_{\tilde{\tau}}T_{d(\mathfrak{\tilde{t}})}(1+T_{i}) &= \left(\sum_{w\in \symgp_{\tilde{\tau}}}T_w\right)(1+T_{i})T_{d(\mathfrak{\tilde{t}})} \\
&= \left(\sum_{w\in \symgp_\nu}T_w\right)T_{d(\mathfrak{\tilde{t}})} = m_\nu T_{d(\mathfrak{\tilde{t}})},\end{align*}
where $\nu$ is the composition of $n$ given by:
\[\nu = (\tau_1, \dots, \tau_s, 1^{l+(i-a)-1}, 2, 1^{m-(i-a)-1}).\]
Let $\lambda = (\tau_1, \dots, \tau_s, 2, 1^{l+m-2})$, the partition of $n$ gained by reordering $\nu$. As $m_\nu = m_{\mathfrak{t}^\nu\mathfrak{t}^\nu} + \check{\hecke}^{\tilde{\tau}}$, we can apply~\cite[Lemma 3.10]{mathas} to write $m_{\mathfrak{t}^\nu\mathfrak{t}^\nu}$ as an $F$-linear combination of elements of the form $m_{\mathfrak{u}\mathfrak{v}}$ where $\mathfrak{u}, \mathfrak{v}\in \std{\lambda}$. Since $\lambda \rhd \tilde{\tau}$, these elements lie in $\check{\hecke}^{\tilde{\tau}}$, and hence again by~\cite[Lemma 3.10]{mathas}, $m_\nu \in\check{\hecke}^{\tilde{\tau}}$. Thus $m_{\tilde{\tau}}T_{d(\mathfrak{\tilde{t}})}(1+T_{i})\in \check{\hecke}^{\tilde{\tau}}$, and therefore in $S^{\tilde{\tau}}$, $m_{\tilde{\mathfrak{t}}}T_{i} = -m_{\tilde{\mathfrak{t}}}$.
\end{itemize}
So in all possible cases we have shown that $\phi((n_\mathfrak{t}\otimes \epsilon)T_{i}) = \phi(n_\mathfrak{t}\otimes \epsilon)T_{i}$, and hence $S^{\tau, m} \cong S^\tau\otimes S^{1^m}$ as $\hecke_\mu$-modules.
\end{proof}
Recall the following version of the Littlewood--Richardson rule from~\cite[13.7]{goldschmidt}. Let $\pi \vdash n$, $n = a + m$, $\mu = (a,m)\vDash n$, and suppose that $\hecke_n$ is semi-simple. Then as $\hecke_\mu$-modules:
\[S^\pi = \bigoplus_{\lambda, \nu} (S^\lambda \otimes S^\nu)^{\oplus c^\pi_{\lambda\nu}},\]
where the sum is over all $\lambda \vdash a$ and $\nu\vdash m$, and $c^\pi_{\lambda\nu}$ are the Littlewood--Richardson coefficients for $\symgp_n$. Our ultimate goal in this section is to show that as $\hecke_\mu$-modules, $S^\tau\otimes S^{(1^m)}$ is a direct summand of $S^{\tilde{\tau}}$ for any Hecke algebra, not just the semi-simple ones. Computing the relevant Littlewood--Richardson coefficients with~\cite[Theorem 4.94]{sagan} gives us this result when $\hecke_n$ is semi-simple.
\begin{lemma}
\label{lrrcoeffs}
Let $\tau\vdash a$ and $\tilde{\tau} \vdash a + m$ as before. Then for $\nu \vdash m$:
\[c^{\tilde{\tau}}_{\tau\nu} = \begin{cases} 1 & \text{ if }\nu = (1^m), \\ 0 & \text{otherwise.} \end{cases}\]
\end{lemma}
We now tackle the general case.
\begin{theorem} \label{specht_rest}
Let $n = a + m$, $\tau \vDash a$ an $e$-core and $\mu = (a,m)\vDash n$. Then as $\hecke_\mu$-modules:
\[S^\tau\otimes S^{(1^m)} \mid S^{\tilde{\tau}}.\]
\end{theorem}
\begin{proof} 
Let $\mathcal{O}$ be the localization of $F[x]$ at the maximal ideal generated by $(x-q)$ and $K$ the field of fractions of $\mathcal{O}$. Consider three related Hecke algebras $\hecke_a(K, x)$, $\hecke_a(\mathcal{O}, x)$ and $\hecke_a(F, q)$. As $K$ is a field, and $x$ has quantum characteristic zero, (and thus each partition is its own $0$-core), by~\cite[Corollary 2.21]{mathas}, $\hecke_a(K,x)$ is semi-simple. As in~\cite[\S 5]{djblocks}, we have an inclusion homomorphism between $\hecke_a(\mathcal{O},x)$ and $\hecke_a(K,x)$, induced by the inclusion of $\mathcal{O}$ into $K$, and a map: 
\[\overline{{}\cdot{}}: \hecke_a(\mathcal{O},x)\to \hecke_a(F,q)\]
induced by $x \mapsto q$. We use the notation $S^\nu_K$ to mean the Specht module corresponding to $\nu$ in $\hecke_a(K,x)$, and similarly for $\mathcal{O}$ and $F$. Following the notation in~\cite[\S 5]{djblocks}, we can define idempotents $H^{b}$ in $\hecke_a(K,x)$, labelled by the blocks of $\hecke_a(F,q)$ (i.e. representatives of tableau which have the same $e$-core), which act as the identity on Specht modules in that block, and zero on all the other Specht modules. As $\tau$ is an $e$-core, and as such is the only Specht module in its block, denote the idempotent corresponding to this block as $H^\tau$. Therefore for $\nu \vdash a$:
\[S^\nu_K H^\tau = \begin{cases} S^\tau_K & \text{if }\nu = \tau \\ 0 & \text{otherwise.} \end{cases}\]
Combining this with the Littlewood--Richardson rule and Lemma~\ref{lrrcoeffs}:
\[S^{\tilde{\tau}}_K(H^\tau\otimes 1) = S^\tau_K \otimes S^{(1^m)}_K.\]

By~\cite[Theorem 5.3]{djblocks}, we know that $H^b \in \hecke_a(\mathcal{O},x)$ for any block $b$ of $\hecke_a(F,q)$, (even though it is defined in $\hecke_a(K,x)$). Furthermore, by~\cite[Theorem 5.4]{djblocks}:
\[\{\overline{H^b}: b\text{ is a block of }\hecke_a(F,q)\}\]
is a complete set of central orthogonal primitive idempotents of $\hecke_a(F,q)$, i.e. $\overline{H^b}$ is the block idempotent of $b$. Therefore $\overline{H^\tau}$ acts as the identity on $S_F^\tau$, and 0 on all other Specht modules. By Lemma~\ref{sub1} and Lemma~\ref{sub2}:
\begin{equation} 
\label{emod}
S^\tau_F \otimes S^{(1^m)}_F = \left(S^\tau_F \otimes S^{(1^m)}_F\right)(\overline{H^\tau}\otimes 1) \subseteq S^{\tilde{\tau}}_F(\overline{H^\tau}\otimes 1). 
\end{equation}
For simplicity of notation, let $V = S_K^{\tilde{\tau}}(H^\tau\otimes 1)$, a $\hecke_\mu(K,x)$-module and $M = S^{\tilde{\tau}}_\mathcal{O}(H^\tau\otimes 1)$ a $\hecke_\mu(\mathcal{O},x)$-module. As $\mathcal{O}$ is a principal ideal domain, and $M$ is an $\mathcal{O}$-submodule of the finite-dimensional $\mathcal{O}$-module $S^{\tilde{\tau}}_\mathcal{O}$, it must have a finite $\mathcal{O}$-basis.

In particular, as $S^{\tilde{\tau}}_\mathcal{O} \otimes_{\mathcal{O}} K \cong S^{\tilde{\tau}}_K$ as $\hecke_{\mu}(K,x)$-modules, and $H^\tau$ is central in both $\hecke_\mu(K,x)$ and $\hecke_\mu(\mathcal{O},x)$, we get that $M\otimes_\mathcal{O} K \cong V$. Using the relevant analogue of~\cite[Proposition 16.12]{curtisreiner1}, we get that $M$ is a free $\hecke_\mu(\mathcal{O},x)$-lattice in $V$, as defined in~\cite[\S 16]{curtisreiner1}. In particular each $\mathcal{O}$-basis of $M$ is a $K$-basis of $V$. Hence:
\[\dim_\mathcal{O}(M) = \dim_K(V).\]
Note that as $\overline{S^{\tilde{\tau}}_\mathcal{O}} = S^{\tilde{\tau}}_F$, and as reducing modules via the map $\overline{{}\cdot{}}$ commutes with multiplication from the Hecke algebra, that:
\[\overline{M} = \overline{S^{\tilde{\tau}}_\mathcal{O}(H^\tau\otimes 1)} = \overline{S^{\tilde{\tau}}_\mathcal{O}}(\overline{H^\tau\otimes 1}) = S^{\tilde{\tau}}_F(\overline{H^\tau}\otimes 1).\]
By the discussion preceding~\cite[Proposition 16.16]{curtisreiner1}:
\[\dim_F(\overline{M}) = \dim_\mathcal{O}(M),\]
therefore:
\[\dim_F(S^{\tau}_F\otimes S^{(1^m)}_F) = \dim_K(S^{\tau}_K\otimes S^{(1^m)}_K) = \dim_K(S^{\tilde{\tau}}_K(H^\tau\otimes 1)) = \dim_F(S^{\tilde{\tau}}_F(\overline{H^\tau}\otimes 1) ).\]
Coupling this with (\ref{emod}) shows that $S^{\tau}_F\otimes S^{(1^m)}_F = S_F^{\tilde{\tau}}(\overline{H^{\tau}}\otimes 1)$, and hence is a direct summand of $S_F^{\tilde{\tau}}$ as a $\hecke_\mu$-module.
\end{proof}
As a result of this theorem, we know that if $S^\tau\otimes S^{(1^{de})}$ has a Green correspondent $M$ in $\hecke_n$, then $M$ will lie in the block $B_{\tau,d}$ as it is a direct summand of $S^{\tilde{\tau}}$ by the module version of Corollary~\ref{maps_corol}. Thus if we can show that this Green correspondent, and the Brauer correspondent of $B_{\tau,0}\otimes B_{\emptyset,d}$ both exist, then we will have successfully identified $B_{\tau,d}$ as the Brauer correspondent of $B_{\tau, 0}\otimes B_{\emptyset, d}$.

%% file: section6.tex
\section{Blocks in characteristic zero}

Throughout this section we will assume that the (algebraically closed) field $F$ has characteristic $0$, hence by the definition in the introduction, an $e$-$0$-parabolic subgroup, (or just $e$-parabolic) is any parabolic subgroup isomorphic to a product of copies of $\symgp_e$. 

\subsection{Vertices of Sign Modules}

We start by looking at the vertex of the sign module, in order to get a lower bound for the vertex of blocks of $\hecke_n$ with empty core. A key tool in characteristic zero is the following theorem~\cite[Theorem 3.1]{du}:
\begin{theorem}
\label{duzero}
If $M$ is a finitely generated indecomposable $\hecke_n$-module, then its vertex is an $e$-parabolic subgroup of $\symgp_n$.
\end{theorem} 
In particular, this means that the sign module $S^{(1^e)}$ for $\hecke_e$ either has fixed-point-free vertex $\symgp_e$ or is projective as an $\hecke_e$-module. We first prove a more general result about $e$-restricted partitions. Recall from~\cite[Proposition 2.11, Theorem 3.43]{mathas} that the non-isomorphic simple modules for $\hecke_n$ are given by:
\[\{D^\lambda = S^\lambda/J(S^\lambda) : \text{ $\lambda\vdash n$ is $e$-restricted}\},\]
where for an $\hecke_n$-module $M$, we denote by $J(M)$ its Jacobson radical. If $\lambda\vdash n$ is an $e$-core, then it is $e$-restricted, and as both $S^\lambda$ and hence $D^\lambda$ lie in a block of weight 0, we conclude by~\cite[Theorem 1.2]{erdmann} that $D^\lambda$ is projective. The next lemma shows the opposite is true when $\lambda$ is not an $e$-core.
\begin{lemma}
Let $\lambda\vdash n$ be $e$-restricted, but not an $e$-core. Then $D^\lambda$ is not projective.
\end{lemma}
\begin{proof}
Note that as $\lambda$ is $e$-restricted, $D^\lambda$ is a non-zero simple module. Denote $d_{\mu\lambda} = [S^\mu: D^\lambda]$ for $\mu\vdash n$, and  assume that $D^\lambda$ is projective. Then in particular, $P^\lambda = S^\lambda = D^\lambda$, where $P^\lambda$ is the corresponding projective indecomposable module. Thus $[P^\lambda:D^\lambda] = 1$. From~\cite[Theorem 2.20]{mathas}, we have:
\[[P^\lambda:D^\lambda] = \sum_{\mu\vdash n} d_{\mu\lambda}^2 \geq 1\]
as $d_{\lambda\lambda} = 1$. Let $\nu$ be another partition in the same block as $\lambda$ (this exists as $\lambda$ is not an $e$-core). Then by~\cite[Corollary 2.22]{mathas}, $S^{\lambda}$ and $S^{\nu}$ are \emph{cell-linked}, i.e.\ there exists a chain of cell-modules $S^{\lambda_i}$ for $i = 0, \dots t$ with $S^{\lambda_0} = S^\lambda$, $S^{\lambda_t} = S^\nu$, and $S^{\lambda_i}$ and $S^{\lambda_{i+1}}$ share a simple composition factor. As $S^\lambda$ and $S^\nu$ are cell-linked, there exists some $S^\mu$ which shares a simple composition factor with $S^\lambda$, which must be $D^\lambda$ itself. Thus $d_{\mu\lambda} \geq 1$, and hence $[P^\lambda: D^\lambda] \geq 2$ giving the required contradiction.
\end{proof}
\begin{corollary}
\label{se_vertex}
$S^{(1^e)}$ has vertex $\symgp_e$ as an $\hecke_e$-module.
\end{corollary}
\begin{proof}
By Theorem~\ref{duzero}, $S^{(1^e)}$ is either projective or has vertex $\symgp_e$. As $S^{(1^e)} = D^{(1^e)}$, it cannot be projective by the preceding lemma as it is not an $e$-core.
\end{proof}
We can now extend this result to larger Hecke algebras. 
\begin{theorem} 
\label{sign_zero}
Let $\lambda = (e^d) \vdash de$ for $d\geq 1$. As a $\hecke_{de}$-module, $S^{(1^{de})}$ has vertex $\symgp_\lambda$.
\end{theorem}
\begin{proof} 
Note that $\symgp_{\lambda} \cong \prod_{i=1}^d\symgp_e$, and as $\hecke_\lambda$-modules we have
\[S^{(1^{de})}\cong (S^{(1^e)})^{\otimes d}.\]
The latter has vertex  $\symgp_\lambda$ as a $\hecke_\lambda$-module by repeated applications of Theorem~\ref{vertex_tensor}. 
As we know the vertex of $S^{(1^{de})}$ is $e$-parabolic, it must be contained in $\symgp_\lambda$. By~\cite[Lemma 3.2]{du}, as it is simple both as a $\hecke_n$ and $\hecke_\lambda$-module, they share the same vertex.
\end{proof}

Thus this gives a lower bound on the vertex of the empty core block of weight $d$ of $\hecke_{de}$ by Theorem~\ref{blocks_and_modules}. We will now give a upper bound by showing that this block is in fact $(\symgp_\lambda, \symgp_\lambda)$-projective, for $\lambda = (e^d)$.

\subsection{Relative projectivity of the empty core block}

We begin with the following definition (from~\cite[\S2 ]{jones}).
\begin{definition}
\label{map/trace_def}
Let $\lambda, \mu \vDash n$ with $\symgp_\lambda \subseteq \symgp_\mu \subseteq \symgp_n$, and $M$ be a $(\hecke_n, \hecke_n)$-bimodule. For $m\in M$, define the \emph{relative trace} from $\symgp_\lambda$ to $\symgp_\mu$ of $m$ as
\[\trace_\lambda^\mu: M \to M,\]
\[m \mapsto\sum_{w\in \mathscr{R}_\lambda^\mu} q^{-\ell(w)}T_{w^{-1}}mT_w.\]
For right $\hecke_n$-modules $P$ and $Q$, we say that $\phi\in \text{Hom}_{\hecke_n}(P, Q)$ is \emph{$\symgp_\lambda$-projective} if there exists $\psi\in \text{Hom}_{\hecke_\lambda}(P, Q)$ such that $\phi = \trace_\lambda^n(\psi)$. Note we can apply the trace map to $\psi$ as since $P$ and $Q$ are right $\hecke_n$-modules, we can view $\text{Hom}_{\hecke_n}(P, Q)$ as a $(\hecke_n, \hecke_n)$-bimodule with $F$-submodule $\text{Hom}_{\hecke_\lambda}(P, Q)$.
\end{definition}
The following from~\cite[Proposition 2.13, Theorem 2.34]{jones} are key facts about the relative trace map.
\begin{itemize}
\item For $\gamma \vDash n$, and $M$ a $(\hecke_n, \hecke_n)$-bimodule, denote
\[Z_M(\hecke_\gamma) = \{m\in M : ma = am \text{ for all } a\in \hecke_\gamma\}.\]
Then $m\in Z_M(\hecke_\gamma)$ implies that $\trace_\gamma^n(m)\in Z_M(\hecke_n)$.
\item A $\hecke_\mu$ module $M$ is relatively $\symgp_\lambda$-projective if and only if the identity map on $M$ as a $\hecke_\mu$-module is $\symgp_\lambda$-projective.
\end{itemize}
As before, let $\lambda = (e^d)\vdash de$. Here we will show that $\hecke_{de}$ has vertex $(\symgp_\lambda, \symgp_\lambda)$, and hence give an upper bound for the vertex of the empty core block.
\begin{theorem}
\label{hde_proj}
$\hecke_{de}$ is relatively $(\symgp_\lambda, \symgp_\lambda)$-projective as a $(\hecke_{de},\hecke_{de})$-bimodule.
\end{theorem}
\begin{proof} We will show as bimodules that:
\[\hecke_{de} \mid \hecke_{de}\otimes_{\hecke_\lambda}\hecke_{de}\cong  \hecke_{de}\otimes_{\hecke_\lambda} \hecke_\lambda \otimes_{\hecke_{\lambda}}\hecke_{de}.\] 
To do this, we define $(\hecke_{de},\hecke_{de})$-bimodule homomorphisms $\varphi:\hecke_{de}\to \hecke_{de}\otimes_{\hecke_\lambda}\hecke_{de}$ and $\psi: \hecke_{de}\otimes_{\hecke_\lambda}\hecke_{de}\to \hecke_{de}$ such that $\psi\varphi = \mathbbm{1}_{\hecke_{de}}$.

As $1\otimes 1 \in Z_{\hecke_{de}\otimes_{\hecke_\lambda} \hecke_{de}}(\hecke_\lambda)$ (as we can push elements of $\hecke_\lambda$ across the tensor product), we can define:
\[x := \trace_\lambda^{de} (1\otimes 1) = \sum_{w\in \mathscr{R}_{\lambda}^{(de)}} q^{-\ell(w)}T_{w^{-1}}\otimes T_w,\]
with $x\in Z_{\hecke_{de}\otimes_{\hecke_\lambda}\hecke_{de}}(\hecke_{de})$ by the above properties. Thus we have a $(\hecke_{de}, \hecke_{de})$-bimodule homomorphism $\varphi:\hecke_{de}\to \hecke_{de}\otimes_{\hecke_\lambda}\hecke_{de}$ given by:
\[h\mapsto hx = xh.\]
Now define 
\[\tilde{x} := \trace_\lambda^{(de)}(1) = \sum_{w\in \mathscr{R}_{\lambda}^{(de)}}q^{-\ell(w)}T_{w^{-1}}T_w\]
By~\cite[Theorem 2.7]{du}, $\tilde{x}$ is invertible, and $\tilde{x}\in Z(\hecke_{de})$ (again by~\cite[Proposition 2.13]{jones}). As $\tilde{x}$ is central, so is $\tilde{x}^{-1}$. Now we can define a $(\hecke_{de}, \hecke_{de})$-bimodule homomorphism $\psi: \hecke_{de}\otimes_{\hecke_\lambda} \hecke_{de}\to \hecke_{de}$ via:
\[a\otimes b \mapsto ab\tilde{x}^{-1}\]
extended linearly, for $a,b\in \hecke_{de}$.

Finally, we show that $\psi\circ \varphi$ is the identity map on $\hecke_{de}$. Note that by the definition of both $x$ and $\tilde{x}$, we have $\psi(x) = \tilde{x}\tilde{x}^{-1} = 1$. Thus:
\[\psi\varphi(h) = \psi(hx) = h\psi(x) = h\cdot 1 = h,\]
completing the proof.
\end{proof}
\begin{corollary}
\label{upper_bound0}
Let $B$ be a block of $\hecke_{de}$, and $\lambda = (e^d)\vdash de$. Then $B$ is relatively $(\symgp_\lambda, \symgp_\lambda)$-projective as a $(\hecke_{de},\hecke_{de})$-bimodule.
\end{corollary}
At this point, we have all the machinery required to show that $(\symgp_\lambda, \symgp_\lambda)$ is the vertex of the empty-core block of $\hecke_{de}$ when our field has characteristic zero. However, we defer the proof to Section 8, where we can cover all cases on the characteristic of $F$ at once.

%% file: section7.tex
\section{Blocks in prime characteristic}

Throughout this section, let $F$ have prime characteristic $p > 0$. Recall that when $e$ is non-zero, then either $(e,p) = 1$ and $q$ is a primitive $e$-th root of unity, or $e = p$ and $q = 1$.

\subsection{Vertices of sign modules}

Our first aim is to prove a lower bound for the vertex of an empty core block of $\hecke_n$. We again do this by considering the vertex of the sign module, and using Theorem~\ref{blocks_and_modules}. Let $\tau \vDash n$, and define:
\[N_\tau  = \sum_{w\in \mathscr{R}_\tau^{(n)}}q^{-\ell(w)}.\]
\begin{proposition} 
\label{useful_calculation}
Let $\tau\vDash n$. Then as a right module, $S^{(1^n)}$ is $\symgp_\tau$-projective if and only $N_\tau\neq 0$.
\end{proposition}
\begin{proof}
Suppose $N_\tau \neq 0$, so is invertible in $F$. Denote the identity map on $S^{(1^n)}$ as a $\hecke_n$-module by $\mathbbm{1}_n$ and as a $\hecke_\tau$-module by $\mathbbm{1}_\tau$. Let $S^{(1^n)}$ be generated by the element $\epsilon$. Then:
\begin{align*} \trace_\tau^{(n)}\left(\frac{1}{N_\tau}\mathbbm{1}_\tau\right)(\epsilon) &= \frac{1}{N_\tau} \sum_{w\in \mathscr{R}_\tau^{(n)}} q^{-\ell(w)}\epsilon\cdot T_{w^{-1}} \mathbbm{1}_\tau T_w \\
										   &= \frac{1}{N_\tau} \sum_{w\in \mathscr{R}_\tau^{(n)}} q^{-\ell(w)}(-1)^{\ell(w^{-1})}\epsilon \cdot T_w \\
										   &= \frac{1}{N_\tau} \sum_{w\in \mathscr{R}_\tau^{(n)}} q^{-\ell(w)}(-1)^{\ell(w^{-1})+\ell(w)}\epsilon\\
										&= \frac{1}{N_\tau} N_\tau \epsilon \\
										&= \epsilon.
\end{align*}
Hence $\trace_\tau^{(n)}(\frac{1}{N_\tau} \mathbbm{1}_\tau) = \mathbbm{1}_n$. Therefore by the remarks following Definition~\ref{map/trace_def}, we conclude that $S^{(1^n)}$ is $\symgp_\tau$-projective.

Now suppose that $S^{(1^n)}$ is $\symgp_\tau$-projective. Again using the aforementioned remarks, there exists a $\hecke_\tau$-homomorphism $\psi$ such that $\mathbbm{1}_n = \trace_{\tau}^{(n)}(\psi)$. Since $S^{(1^n)}$ is an irreducible $\hecke_\tau$-module,  $\psi = f\mathbbm{1}_\tau$ for some $f\in F$. Then the above calculation shows that:
\[\mathbbm{1}_n = \trace_\tau^{(n)}(f \mathbbm{1}_\tau) = f N_\tau \mathbbm{1}_n\]
hence $f N_\tau = 1$, so $N_\tau$ must be non-zero.
\end{proof}
Therefore relative projectivity of $S^{(1^n)}$ relies entirely upon these $N_\tau$. Consider the following polynomial in $(\mathbb{Z}/p\mathbb{Z})[u]$: 
\[P_\tau^n (u) = \sum_{w\in \mathscr{R}_\tau^{(n)}} u^{\ell(w)},\]
and notice that $N_\tau = P_\tau^n (q^{-1})$. By~\cite[\S 1.11]{humphreys}, $P_\tau^n = P_n/P_\tau$ where $P_n$ is the Poincar\'e polynomial of $\symgp_n$, and $P_\tau$ is the Poincar\'e polynomial of $\symgp_\tau$. Thus to check relative projectivity of $S^{(1^n)}$, it suffices to count the zeroes of $P_n$ and $P_\tau$ at $q^{-1}$. 
\begin{definition}
For $q$ a primitive $e$-th root of unity in $F$ (or $q = 1$ if $e = p$) and $P\in F[u]$, define $z(P)$ to be largest integer $l$ such that $(u-q^{-1})^l \mid P(u)$ in $F[u]$.
\end{definition}
Thus we have the following test:
\begin{corollary} 
\label{proj_test}
For $\tau\vDash n$, $N_\tau\neq 0$ if and only if $z(P_n) = z(P_\tau)$. Hence $S^{(1^n)}$ is $\symgp_\tau$-projective if and only if $z(P_n) = z(P_\tau)$
\end{corollary}
From~\cite[\S 2]{dipperdu} we know that:
\[P_n(u) = \prod_{i=1}^n \frac{u^i-1}{u-1} = \prod_{i=2}^n (1+\dots + u^{i-1}).\]
We also know that for any $i$
\[u^i-1 = \prod_{d\mid i} \Phi_d(u)\]
where $\Phi_d$ is the $d$-th cyclotomic polynomial. Now denote: 
\begin{equation}
\label{qdef}
Q_i(u) :=1+\dots+u^{i-1} = \prod_{d\mid i, d>1} \Phi_d(u), 
\end{equation} 
so that $P_n(u) = \prod_{i=2}^n Q_i(u)$. As we can write each $P_n$ as a product of cyclotomic polynomials, we only need to compute $z(\Phi_m)$ for $\Phi_m$ involved in $P_n$. 

\subsubsection{Resultants and zeroes of cyclotomic polynomials}

Recall the notion of the resultant $\rho(f,g)$ of two polynomials $f,g\in R[x]$ for some ring $R$, see for example~\cite[\S 2]{apostol}. This has the property that $\rho(f,g) = 0$ if and only if $f$ and $g$ share a common factor. Using~\cite[Theorems 3 and 4]{apostol}, we can compute the resultant of two cyclotomic polynomials. We reproduce these results below. Without loss of generality let $m > n > 1$. Then:
\[ \rho(\Phi_m,\Phi_n) = \rho(\Phi_n,\Phi_m) = \begin{cases} s^{\varphi(n)} & \text{ if }m/n\text{ is a power of some prime }s, \\ 1  & \text{otherwise,} \end{cases}\]
where $\varphi$ is Euler's totient function. This allows us to compute $z(\Phi_n)$ for general $n$. 
\begin{theorem} 
\label{cyc_zero}
Let $q$ have quantum characteristic $e$, and let $n > 1$. Then $\Phi_n(q^{-1}) = 0$ if and only if $n = ep^r$ for some $r \geq 0$. In particular:
\begin{itemize}
\item If $(e,p) = 1$, then $z(\Phi_{ep^r}) = p^r-p^{r-1}$ for $r \geq 1$, and $z(\Phi_e) = 1$.
\item If $e = p$ and $q=1$, then $z(\Phi_{p^r}) = p^{r} - p^{r-1}$ for $r \geq 1$.
\end{itemize}
\end{theorem}
\begin{proof}
First of all, if $n < e$, then $z(\Phi_n) = 0$ as $\Phi_e$ is the smallest cyclotomic polynomial which can be zero at $q^{-1}$. Now suppose $\Phi_n(q^{-1}) = 0$, and $n > e$. Consider the resultant of $\Phi_n$ with $\Phi_e$. This resultant must be zero, as $(u-q^{-1})$ is a common factor of both by assumption. As $n > e$, by the above result from~\cite[Theorems 3, 4]{apostol}, we can only have $\rho(\Phi_n, \Phi_e) = 0$ in $F$ if $n/e$ is a power of $p$, i.e.\ $n = ep^r$ for some $r \geq 1$. Including the possibility when $n = e$, gives one direction of our first assertion.

It remains to show that $\Phi_{ep^r}$ are zero at $q^{-1}$ for all $r\geq 0$, and to compute $z(\Phi_n)$ in these cases. Recall from~\cite[\S1 Equations 4,5]{nagell} that:
\[\Phi_{np}(u) = \begin{cases} \Phi_n(u^p)/\Phi_n(u) & \text{ if } p \nmid n, \\
								\Phi_n(u^p) & \text{ if } p \mid n. \end{cases} \]
Thus when $(e,p) = 1$, and $n = ep^r$ for $r \geq 1$:
\[\Phi_n(u) = \Phi_{ep^r}(u) = \Phi_{ep^{r-1}}(u^p) = \cdots = \Phi_{ep}(u^{p^{r-1}}) = \Phi_e(u^{p^r})/\Phi_e(u^{p^{r-1}}).\]
As $F$ has characteristic $p$:
\[\Phi_n(u) = \Phi_e(u^{p^r})/\Phi_e(u^{p^{r-1}}) = \Phi_e(u)^{p^r}/\Phi_e(u)^{p^{r-1}} = \Phi_e(u)^{p^r-p^{r-1}}.\]
Thus as $\Phi_e(q^{-1}) = 0$, we get that $\Phi_n(q^{-1}) = 0$. As $z(\Phi_e) = 1$ (its roots are the primitive $e$-th roots of unity each with multiplicity one), we also get that $z(\Phi_n) =  p^r-p^{r-1}$ if $n = ep^r$ for $r \geq 1$.

Similarly when $e = p$ and $q = 1$ (so $q = q^{-1}$):
\[\Phi_n(u) = \Phi_{p^r}(u) = \Phi_{p^{r-1}}(u^p) = \cdots = \Phi_p(u^{p^{r-1}}) = \Phi_p(u)^{p^{r-1}}.\]
Thus $z(\Phi_{p^r}) = p^{r-1}z(\Phi_p) = p^r - p^{r-1}$ since $z(\Phi_p) = p-1$.
\end{proof}

\subsubsection{Computing with $z(P_n)$}
\label{subsubfun}

We begin by proving the following preliminary expressions.
\begin{lemma} 
\label{qi_1}
Let $i > 1$. Then:
\begin{itemize}
\item If $(e,p) = 1$:
\[z(Q_i) = \begin{cases} p^r & \text{ if }r\text{ is the largest integer such that }ep^r \mid i, \\ 0 & e \nmid i. \end{cases}\]
\item If $e = p$:
\[z(Q_i) = \begin{cases} p^{r} - 1 & \text{ if }r\text{ is the largest integer such that }p^r \mid i, \\ 0 & p \nmid i. \end{cases}\]
\end{itemize}
\end{lemma}
\begin{proof}
This follows from counting the number of zeroes at $q^{-1}$ in the product (\ref{qdef}). When $(e,p) = 1$, then if $e \nmid i$, we have $z(Q_i) = 0$ as no $\Phi_{ep^r}$ appear in the product (\ref{qdef}). Otherwise, if $r$ is the largest integer such that $ep^r \mid i$, then $\Phi_e, \dots, \Phi_{ep^r}$ are the only factors which are zero at $q^{-1}$. Thus:
\[z(Q_i) = 1 + (p - 1) + \dots + (p^r - p^{r-1}) = p^r.\]
If $e = p$, then if $p \nmid i$, there are no zeroes at $q = 1$, otherwise we only have factors $\Phi_p, \dots, \Phi_{p^r}$ which are zero at $q$, where $r$ is the largest integer such that $p^r \mid i$. Thus:
\[z(Q_i) = (p -1) + \dots + (p^r - p^{r-1}) = p^r - 1, \] completing the proof.
\end{proof}
Throughout the rest of this subsection, we will state results for both $(e,p) = 1$ and $e=p$, but will not prove the latter case. This is because the proof follows in exactly the same way, just using the different value of $z(Q_i)$ given above.
\begin{lemma} 
\label{general_1}
Suppose $(e,p) = 1$ and $r$ is the largest integer such that $ep^r \leq n$. Then
\[z(P_n) = \sum_{l = 0}^{r-1} \left( \left\lfloor \frac{n}{ep^l}\right\rfloor - \left\lfloor \frac{n}{ep^{l+1}}\right\rfloor \right)p^l + \left\lfloor \frac{n}{ep^r}\right\rfloor p^r.\]
If $e = p$ and $q = 1$, and $r > 1$ is the largest integer with $p^r \leq n$. Then
\[z(P_n) = \sum_{l=1}^{r-1} \left( \left\lfloor \frac{n}{p^l}\right\rfloor - \left\lfloor \frac{n}{p^{l+1}}\right\rfloor \right)(p^l-1) + \left\lfloor \frac{n}{p^r}\right\rfloor (p^r-1).\]
\end{lemma}
\begin{proof}
Recall that $P_n = \prod_{i=2}^n Q_i$, hence $z(P_n) = \sum_{i=2}^n z(Q_i)$. Now each $Q_i$ contributes either no zeroes if no $ep^l$ divides $i$ or $z(Q_i)$ zeroes if it does. If it contributes zeroes, it contributes according to the largest $l$ such that $ep^l \mid i$. Hence we need to count how many times this occurs. For each $l$, the number of times that $ep^l$ divides $n$ is $\lfloor \frac{n}{ep^l}\rfloor$. In $\lfloor \frac{n}{ep^{l+1}}\rfloor$ of those times, we also have $ep^{l+1}$ dividing $n$. Hence the total number of times $l$ is the largest integer such that $ep^l$ divides $i$ for $i = 2, \dots, n$ is $\lfloor \frac{n}{ep^l}\rfloor - \lfloor \frac{n}{ep^{l+1}}\rfloor$ for $0 \leq l \leq r-1$, or $\lfloor \frac{n}{ep^r}\rfloor$ when $l = r$. Summing all these occurrences of zeroes and using the values from Lemma~\ref{qi_1}, gives the result as required.
\end{proof}
Recall from the introduction that the $e$-$p$-adic expansion of $n\in \mathbb{N}$ is the unique decomposition of $n$ as:
\[ n = a_{-1} + a_0 e + a_1 ep + \dots a_r ep^r \]
where $0 \leq a_{-1} < e$ and $0 \leq a_i < p$ for $i = 0, \dots, r$. If $e = p$, the $e$-$p$-adic expansion is just the usual $p$-adic expansion, and we will simplify notation in this setting by writing 
\[ n = b_{0} + b_1 p + \dots b_r p^r\] where $0 \leq b_i < p$ for $i = 0, \dots, r$. The previous lemma lets us compute $z(P_n)$ based on these expansions:
\begin{theorem}
\label{formula_1}
Suppose $(e,p) = 1$. Let $n > 1$ and write $n = a_{-1} + a_0e + a_1ep + \dots + a_r ep^r$ its $e$-$p$-adic expansion. Then:
\[z(P_n) = a_0 + \sum_{l = 1}^r a_l\Big((l + 1)p^l - lp^{l-1}\Big).\]
Suppose $e = p$ and $q = 1$. Let $n > 1$ and write $n = b_0 + b_1p + \dots b_rp^r$ its $p$-adic expansion. Then:
\[z(P_n) = \sum_{l = 1}^r b_l l (p^l - p^{l-1}).\]
\end{theorem}
\begin{proof}
To get this result from Lemmas~\ref{qi_1} and~\ref{general_1}, we first compute $\lfloor \frac{n}{ep^l}\rfloor - \lfloor \frac{n}{ep^{l+1}}\rfloor$ for $0 \leq l \leq r-1$.
\begin{align*}\left(\left\lfloor \frac{n}{ep^l}\right\rfloor - \left\lfloor \frac{n}{ep^{l+1}}\right\rfloor\right) &= a_l + a_{l+1}p + \dots + a_r p^{r-l} - (a_{l+1} + a_{l+2}p + \dots + a_r p^{r-l-1}) \\
 & = (a_l - a_{l+1}) + (a_{l+1} - a_{l+2})p + \dots (a_{r-1} - a_r)p^{r-l-1} + a_rp^{r-l}.
\end{align*}
Collecting terms by the $a_l - a_{l+1}$ in the sum gives us:
\[z(P_n) = \sum_{l = 0}^{r-1} \Big((a_l - a_{l+1})(z(Q_{ep^l}) + pz(Q_{ep^{l-1}}) + \dots + p^lz(Q_{ep^0})) + a_rp^{r-l}z(Q_{ep^l})\Big) + a_r p^r,\]
and using the fact that $z(Q_{ep^j}) = p^j$ for all $j\geq 0$, this expression simplifies to
\begin{align*}z(P_n) &= \sum_{l=0}^{r-1} \Big((a_l-a_{l+1})(p^l + \dots + p^l) + \dots a_r p^r\Big)  + a_r p^r \\
&= \sum_{l = 0}^{r-1} \Big( (l+1)(a_l - a_{l+1})p^l \Big) + (r+1) a_r p^r\\
&= a_0 - a_1 + \left(\sum_{l = 1}^{r-1} a_l(l + 1)p^l\right) - \left( \sum_{l = 1}^{r-1} a_{l+1}(l + 1)p^l \right) + (r + 1)a_r p^r\\ 
&= a_0 + \sum_{l = 1}^{r} a_l\Big((l + 1)p^l - lp^{l-1}\Big),\end{align*}
if we collect by the coefficients $a_i$.
\end{proof}
\begin{corollary}
\label{sign_n_1}
If $(e,p) = 1$ and $r\geq 0$, then $z(P_{ep^r}) = (r+1)p^r - rp^{r-1}$.
\\
\\ If $e = p$ and $r \geq 1$, then $z(P_{p^r}) = r(p^r - p^{r-1})$.
\end{corollary}
We can use Theorem~\ref{formula_1} to show that if $\lambda$ is the partition corresponding to the standard maximal $e$-$p$-parabolic subgroup of $\symgp_n$, then $S^{(1^n)}$ is $\symgp_\lambda$-projective.
\begin{proposition}
\label{upper_1}
Let $n > 1$, and denote by $\lambda$, the composition of $n$ corresponding to its $e$-$p$-adic expansion. Then $S^{(1^n)}$ is $\symgp_\lambda$-projective. 
\end{proposition}
\begin{proof} Again, we will only prove this when $(e, p) = 1$. We show that $z(P_n) = z(P_\lambda)$. We already have a formula for $z(P_n)$, so we compute $z(P_\lambda)$. As $P_\lambda = \prod_{i=0}^r (P_{ep^i})^{a_i}$:
\begin{align*} z(P_\lambda) &= \sum_{i = 0}^r a_i z(P_{ep^i}) \\
				 &= a_0 + \sum_{i= 1}^ r \Big(a_i (i+1)p^i - a_i i p^{i-1} \Big)\\
				 &= z(P_n). 
				 \end{align*}
Applying Corollary~\ref{proj_test} gives the result.
\end{proof}
So we have obtained an upper bound for the vertex of $S^{(1^n)}$ for general $n$. We now prove the special case of the vertex of $S^{(1^n)}$ where $n = ep^r$ for some $r \geq 0$. By~\cite[Theorem 2.9]{dipperdu} the vertex of $S^{(1^n)}$ is $e$-$p$-parabolic, so these are the only $\tau$ we need to check.
\begin{lemma}
\label{sign_vertex_1}
Suppose either $(e, p) = 1$ and $n = ep^r$ for $r \geq 0$, or $e = p$ and $n = p^r$ for $r > 0$. Then for any $e$-$p$-parabolic subgroup $\symgp_\tau\subsetneq \symgp_n$ of $\symgp_n$, $z(P_n) > z(P_\tau)$. Hence $S^{(1^n)}$ has vertex $\symgp_n$.
\end{lemma}
\begin{proof}
We once again will only prove this statement when $(e,p) = 1$. Let $\symgp_{\tau}$ be the $e$-$p$-parabolic subgroup corresponding to the expression $n = a_{-1} + a_0 e + \dots + a_t ep^t$ for natural numbers $a_i$. This no longer has to be a reduced expression, but as $\symgp_\tau \subsetneq \symgp_n$ we have in particular that $t < r$. Then we have by Corollary~\ref{sign_n_1} that:
\begin{align*} z(P_\tau) &= a_0 + \sum_{i=1}^ t\Big( a_i(i+1) p^i - a_i i p^{i-1}\Big) \\
				&= \sum_{i=0}^t a_i p^i + \sum_{i=1}^ t a_i i p^{i-1}(p-1). \end{align*}
As $n = a_{-1}+ \sum_{i=1}^t a_i ep^i = ep^r$, we get immediately that $\sum_{i=0}^t a_i p^i \leq p^r$, and hence $\sum_{i=1}^t a_i p^{i-1} \leq p^{r-1}$. This tells us that
\begin{align*} z(P_\tau) &= \sum_{i=0}^t a_i p^i + \sum_{i=1}^ t a_i i p^{i-1}(p-1) \\
				&\leq p^r + (p-1)\sum_{i = 1}^t a_i i p^{i-1} \\
                &< p^r + r(p-1)\sum_{i = 1}^t a_i p^{i-1} \\
                &\leq p^r + r(p-1)p^{r-1} \\
                &= z(P_n). \end{align*}
Thus if $\symgp_\tau \subsetneq \symgp_n$, we have $N_\tau$ is zero and the vertex of $S^{(1^n)}$ as an $\hecke_n$-module must be $\symgp_n = \symgp_{ep^r}$.
\end{proof}

\subsubsection{Computing the vertex of $S^{(1^n)}$}

We can now compute the vertex of $S^{(1^n)}$ for $n\geq 1$ in both cases on $e$ and $p$.
\begin{theorem}
\label{sign_vertex}
Let $n\geq 1$. Then $S^{(1^n)}$ has vertex $\symgp_\lambda$ as a $\hecke_n$-module, where $\symgp_\lambda$ is the standard maximal $e$-$p$ parabolic subgroup of $\symgp_n$.
\end{theorem}
\begin{proof} Proposition~\ref{upper_1} gives $\symgp_\lambda$ as an upper bound for the vertex. Now suppose that $S^{(1^n)}$ has vertex $\symgp_\tau$ which is strictly contained in $\symgp_\lambda$. We can assume that $\symgp_\tau$ is $e$-$p$-parabolic by~\cite[Theorem 2.9]{dipperdu}. Then by Corollary~\ref{proj_test}, $z(P_n) = z(P_\tau)$, and in particular $z(P_\lambda) = z(P_\tau)$.

Writing $\lambda = (\lambda_1, \dots, \lambda_s)$, as $\symgp_\tau \subset \symgp_\lambda$, there exist compositions $\tau^{(i)}$ such that $\tau^{(i)} \vDash \lambda_i$ and $\prod_{i=1}^s \symgp_{\tau^{(i)}} \cong \symgp_\tau$. 

For each $i$, as $P_{\lambda_i}/P_{\tau^{(i)}}$ is a non-zero polynomial with coefficients in $\mathbb{Z}/p\mathbb{Z}$, we have $z(P_{\lambda_i}) \geq z(P_{\tau^{(i)}})$. As $\symgp_\tau$ is strictly contained in $\symgp_\lambda$, then there exists some $j$ with $\symgp_{\tau^{(j)}} \subsetneq \symgp_{\lambda_{j}}$. Since $S^{1^{(\lambda_j)}}$ is not $\symgp_{\tau^{(j)}}$-projective by Lemma~\ref{sign_vertex_1}, applying Corollary~\ref{proj_test} tells us that $z(P_{\lambda_j}) > z(P_{\tau^{(j)}})$. Thus $z(P_\lambda) > z(P_\tau)$, giving a contradiction. Hence we must have that $\symgp_\tau$ cannot be strictly contained in $\symgp_\lambda$, and thus $\symgp_\lambda$ must be the vertex of $S^{(1^n)}$ as a $\hecke_n$-module.
\end{proof}

\subsection{Relative projectivity of empty core blocks}

Here we prove an upper bound for the vertex of blocks with empty core. We cannot fully generalise Theorem~\ref{hde_proj}, however we can state a similar theorem which only covers the block itself. Denote the central primitive idempotent associated to the block $B$ by $e_B$, and let $\End{B}{B}$ be the ring of $(B,B)$-bimodule homomorphisms on $B$. This is a local ring by~\cite[Theorem 4.2]{alperin} as $B$ is an indecomposable $(B,B)$-bimodule. Furthermore, as in the proof of~\cite[Lemma 2.3]{du}, $\End{B}{B}\cong Z(B)$, and hence $Z(B)$ is local. Thus $x\in Z(B)$ is invertible if and only if its image $\overline{x}\in Z(B)/J(Z(B))$ is non-zero (in a local ring, the Jacobson radical consists of all the non-units).

As we have a canonical isomorphism $\theta: Z(B)/J(Z(B))\to F$, the action of $Z(B)/J(Z(B))$ on a one-dimensional $Z(B)$-module $M$ must coincide with the action of the field, i.e. for $x\in Z(B)$ and $m\in M$, if $m\overline{x} = \beta m$, then $\theta(\overline{x}) = \beta$. Thus $x$ is invertible in $Z(B)$ if and only if $\beta \neq 0$.

Denote $B = B_{\emptyset, d}$ the block of $\hecke_{de}$ with empty core and $e$-weight $d$, let $\symgp_\lambda$ be the standard maximal $e$-$p$-parabolic subgroup of $\symgp_{de}$, and define
\[x_B := \trace_{\lambda}^{(de)}(e_B) = \sum_{w \in \mathscr{R}_{\lambda}^{(de)}} q^{-\ell(w)} T_{w^{-1}} e_B T_w.\]
\begin{lemma}
$x_B$ is invertible in $Z(B)$, and hence in $B$.
\end{lemma}
\begin{proof}
Take $S^{(1^{de})} = \langle \epsilon \rangle$, the one-dimensional sign $\hecke_{de}$-module. We now compute $\epsilon\cdot x_B$. As multiplication by $e_B$ is the identity map, using the same calculations from the proof of Proposition~\ref{useful_calculation} we obtain
\[ \epsilon\cdot x_B = \trace_{\lambda}^{(de)}(\epsilon) = \left(\sum_{w\in \mathscr{R}_{\lambda}^{(de)}} q^{-\ell(w)} \right) \epsilon.\]
Hence under the isomorphism between $Z(B)/J(Z(B))$ and $F$: 
\[\theta(\overline{x_B}) = \sum_{w\in \mathscr{R}_\lambda^{(de)}}q^{-\ell(w)} = N_\lambda\]
which is non-zero by Proposition~\ref{useful_calculation} and Proposition~\ref{upper_1}. Thus by the preceding discussion, $x_B$ is invertible in $Z(B)$, and hence in $B$.
\end{proof}
We can now generalise the proof of Theorem~\ref{hde_proj} to the characteristic $p$ case, only focusing on the empty core block.
\begin{theorem}
\label{dir_summand_B}
Let $B = B_{\emptyset,d}$ the empty core block of $\hecke_{de}$, and $\lambda\vDash de$ the composition corresponding to the standard maximal $e$-$p$-parabolic subgroup of $\symgp_{de}$. Then as $(\hecke_{de},\hecke_{de})$-bimodules, $B \mid B\otimes_{\hecke_\lambda} B$.
\end{theorem}
\begin{proof}
Define a map $\varphi:B\to B\otimes_{\hecke_\lambda} B$ by $h\mapsto h \trace_{\lambda}^{(de)}(e_B\otimes e_B)$ and $\psi: B\otimes_{\hecke_\lambda} B\to B$ by $a\otimes b \mapsto ab x_B^{-1}$ extended linearly. 

As in the proof of Theorem~\ref{hde_proj}, both are well-defined $(\hecke_{de},\hecke_{de})$-bimodule homomorphisms, and $\psi\circ \varphi = \mathbbm{1}_{\hecke_{de}}$. Thus $B$ is a direct summand of $B\otimes_{\hecke_\lambda} B$ as $(\hecke_{de}, \hecke_{de})$-bimodules.
\end{proof}
\begin{corollary}
\label{upper_bound}
As a $(\hecke_{de}, \hecke_{de})$-bimodule, $B$ is relatively $(\symgp_\lambda, \symgp_\lambda)$-projective.
\end{corollary}
\begin{proof}
By definition, $B \mid \hecke_{de}$ as a $(\hecke_{de}, \hecke_{de})$-bimodule, and therefore as both $(\hecke_\lambda, \hecke_{de})$ and $(\hecke_{de}, \hecke_\lambda)$-bimodules as well. By the previous theorem:
\[B \mid B\otimes_{\hecke_\lambda} B \mid \hecke_{de} \otimes_{\hecke_{\lambda}} B \mid \hecke_{de}\otimes_{\hecke_{\lambda}} \hecke_{de} \cong \hecke_{de}\otimes_{\hecke_{\lambda}}\hecke_{\lambda} \otimes_{\hecke_{\lambda}} \hecke_{de},\] showing $B$ is relatively $(\symgp_\lambda, \symgp_\lambda)$-projective.
\end{proof}

%% file: section8.tex
\section{Computing vertices of blocks}

In the previous sections, we showed that in all characteristics the empty-core block of $\hecke_{de}$ was $(\symgp_\lambda, \symgp_\lambda)$-projective, where $\symgp_\lambda$ is the standard maximal $e$-$p$-parabolic subgroup of $\symgp_{de}$. We also found a module in that block ($S^{(1^{de})}$) which had vertex $\symgp_\lambda$ too. We will first show that $(\symgp_\lambda, \symgp_\lambda)$ is actually the vertex of this block, before applying our Brauer correspondence from Section 4 to compute the vertices of all blocks.
\begin{proposition}
\label{lower_bound}
Let $B = B_{\emptyset, d}$ be the block of $\hecke_{de}$ with empty core, and $\symgp_\lambda$ the standard maximal $e$-$p$-parabolic subgroup of $\symgp_{de}$. Then as a $(\hecke_{de}, \hecke_{de})$-bimodule, $B$ has no vertex strictly contained in $(\symgp_\lambda, \symgp_\lambda)$. 
\end{proposition}
\begin{proof}
Suppose that $B$ has a vertex $(\symgp_{\tau_1}, \symgp_{\tau_2})\subsetneq (\symgp_\lambda, \symgp_\lambda)$. By Corollary~\ref{block_test}, as $S^{(1^{de})}$ lies in this block and has vertex $\symgp_\lambda$ as a right $\hecke_{de}$-module (by Corollary~\ref{sign_zero} or Theorem~\ref{sign_vertex}), there must be some $g\in \symgp_n$ with $\symgp_\lambda^g \leq \symgp_{\tau_2} \leq \symgp_\lambda$, thus $\symgp_{\tau_2} = \symgp_\lambda$.

By earlier assumption, $\symgp_{\tau_1}\subsetneq \symgp_\lambda$. In particular, $B$ is $(\symgp_{\tau_1}, \symgp_n)$-projective and hence by Proposition~\ref{argue_down}, it is also $(\symgp_{\tau_1}, \symgp_{\tau_1})$-projective. This means that $B$ has a vertex which whose right vertex is contained in $\symgp_{\lambda}$. This cannot happen by the preceding argument, so $B$ has no vertex strictly contained within $(\symgp_\lambda, \symgp_\lambda)$.
\end{proof}
\begin{theorem}
\label{empty_core_vertex} Let $b$ be the block of $\hecke_{de}$ with empty core and $e$-weight $d$. Then $b$ has vertex $(\symgp_\lambda, \symgp_\lambda)$ where $\lambda$ is the composition of $de$ corresponding to the standard maximal $e$-$p$-parabolic subgroup of $\symgp_{de}$.
\end{theorem}
\begin{proof}
By Corollary~\ref{upper_bound0} or Corollary~\ref{upper_bound} (depending on the characteristic), $b$ is $(\symgp_\lambda, \symgp_\lambda)$-projective, and hence has a vertex contained in $(\symgp_\lambda, \symgp_\lambda)$. Proposition~\ref{lower_bound} says $b$ cannot have a vertex strictly contained in $(\symgp_\lambda, \symgp_{\lambda})$, finishing the proof.
\end{proof}
\begin{proposition} 
\label{vertices_of_mu} Let $\rho$ be an $e$-core, $\mu = (\lvert \rho\rvert,de)\vDash n$, $\tau = (1^{\lvert \rho\rvert}, de)$, and $\symgp_\lambda$ the standard maximal $e$-$p$-parabolic subgroup of $\symgp_\tau$. Let $b_{\rho, 0}$ be the block of $\hecke_{|\rho|}$ corresponding to $\rho$, and $b_{\emptyset, d}$ the block of $\hecke_{de}$ with empty-core. Denote $b := b_{\rho, 0}\otimes b_{\emptyset, d}$ a block of $\hecke_\mu = \hecke_{|\rho|}\otimes \hecke_{de}$. Then $b$ has vertex $(\symgp_\lambda, \symgp_\lambda)$, and thus $b^{\hecke_n}$ exists.
\end{proposition}
\begin{proof}
Since blocks of $e$-weight $0$ are projective (they are semi-simple from~\cite[Theorem 1.2]{erdmann}), as a $(\hecke_\mu, \hecke_\mu)$-bimodule, $b$ has vertex $(\symgp_\lambda, \symgp_\lambda)$ by Theorem~\ref{empty_core_vertex} and Theorem~\ref{vertex_tensor}. Then $b^{\hecke_n}$ exists by Theorem~\ref{new_brauer} (as $e\mid de$, and $\symgp_\lambda$ is a fixed-point-free subgroup of $\symgp_\tau$).
\end{proof}
So we have shown that there exists a block of $\hecke_n$ with vertex $(\symgp_\lambda, \symgp_\lambda)$. We now need to identify this block, and show that all blocks can be found in this way.
\begin{theorem}[Classification of vertices of blocks of Hecke algebras]
\label{classification} Let $\rho$ be an $e$-core, $\mu = (\lvert \rho\rvert,de)\vDash n$, $\tau = (1^{\lvert \rho\rvert}, de)$, and $\symgp_\lambda$ the standard maximal $e$-$p$-parabolic subgroup of $\symgp_\tau$. Denote $B=B_{\rho, d}$, the block of $\hecke_n$ with $e$-core $\rho$ and $e$-weight $d$. Then $B$ has vertex $(\symgp_\lambda, \symgp_\lambda)$ as a $(\hecke_n, \hecke_n)$-bimodule.
\end{theorem}
\begin{proof} 
When $d = 0$, our block is semi-simple by~\cite[Theorem 1.2]{erdmann}, and thus is projective as a bimodule over itself and hence has trivial vertex as required. Now suppose $d > 0$. 

Consider the block $b_{\rho,0}\otimes b_{\emptyset,d}$ of $\hecke_\mu$. By the previous proposition this has vertex $(\symgp_{\lambda}, \symgp_\lambda)$, and has a Brauer correspondent; we will show that this is $B_{\rho,d}$, by applying Corollary~\ref{block_test}. 

$S^\rho \otimes S^{(1^{de})}$ is an indecomposable module which lies in $b$ with vertex $\symgp_\lambda$ by Theorem~\ref{sign_zero} or Theorem~\ref{sign_vertex}. Applying~\cite[Theorem 3.6]{du}, it has a Green correspondent $M$ in $\hecke_n$. By Theorem~\ref{specht_rest} $S^\rho \otimes S^{(1^e)} \mid S^{\tilde{\rho}}$ as $\hecke_\mu$-modules, so applying Corollary~\ref{maps_corol} (with $\sigma_1 = (1)$ and $\sigma_2 = (n)$), tells us that $M \mid S^{\tilde{\rho}}$ as $\hecke_n$-modules, thus $M$ lies in $B_{\rho,d}$. As such, we conclude with Corollary~\ref{block_test} that $B = b^{\hecke_n}$ and hence has vertex $(\symgp_\lambda, \symgp_\lambda)$.
\end{proof}

%% file: section9.tex
\section{The Dipper--Du conjecture}

One application of our classification of blocks, is resolving the Dipper--Du conjecture given in the introduction This was first stated as~\cite[Conjecture 1.9]{dipperdu}, and shown to be true for Young modules in~\cite[\S 5]{dipperdu}, for fields of characteristic zero in~\cite[Theorem 3.1]{du}, and in blocks of $e$-weight 1 in~\cite[Theorem 18.1.13]{schmider}. Note that in~\cite{counterexample}, a supposed counter-example was given to this conjecture when $p = 2$ and $e = 3$. Here, an indecomposable $\hecke_3$-module $M$ is found, which is $\hecke_{(2,1)}$-projective as a $\hecke_3$-module. However, as $\hecke_{(2,1)}$ is semi-simple when $e=3$, $M$ is a projective $\hecke_{(2,1)}$-module, and hence by Corollary~\ref{transitivity}, is projective as a $\hecke_3$-module. This contradicts the earlier statement in~\cite{counterexample} that $M$ could not be projective~\cite[Theorem 2.2 Part (2)]{counterexample}. We are able to use our classification to prove this conjecture:
\begin{theorem}
Let $F$ be an (algebraically closed) field of characteristic $p > 0$, $n\in \mathbb{N}$, and $q\in F$ a primitive $e$-th root of unity. Then the vertices of indecomposable $\hecke_n$-modules are $e$-$p$-parabolic.
\end{theorem}
\begin{proof}
Let $M$ be an indecomposable (right) $\hecke_n$-module with vertex $\symgp_\tau$, where $\tau = (\tau_1, \dots, \tau_s) \vDash n$. By~\cite[Lemma 3.2]{du}, there is an indecomposable $\hecke_\tau$-module $N$ such that $M \mid N\otimes_{\hecke_\tau} \hecke_n$ and $N$ has vertex $\symgp_\tau$. As $N$ is indecomposable, $N$ must belong to a block $b$ of $\hecke_\tau$, with
\[b = b_{\rho_1, d_1}\otimes \cdots\otimes b_{\rho_s, d_s},\]
where $b_{\rho_i,d_i}$ is the block of $\hecke_{\tau_i}$ corresponding to $e$-core $\rho_i$ and $e$-weight $d_i$. By Theorem~\ref{classification}, $b$ has vertex $(\symgp_\lambda, \symgp_\lambda)$ where $\symgp_\lambda \cong \symgp_{\lambda^1}\times \dots\times\symgp_{\lambda^s}$, and $\symgp_{\lambda^i}$ is the standard maximal $e$-$p$-parabolic subgroup of $\symgp_{(1^{\lvert \rho_i \rvert})}\times \symgp_{d_ie}\subset \symgp_{\tau_i}$. As $N$ lies in the block $B$, we must have that $\symgp_\tau\subset_{\symgp_\tau} \symgp_\lambda$, and thus by Theorem~\ref{blocks_and_modules},  $\symgp_\lambda = \symgp_\tau$. In particular, for each $i$, we get $\symgp_{\lambda^i} = \symgp_{\tau_i}$.

Thus each $(\tau_i)\vDash \tau_i$ is an $e$-$p$-parabolic composition, so either $\tau_i = ep^r$ for some $r\geq 0$, or $\tau_i = 1$.

Hence $\tau = (\tau_1, \dots, \tau_s)$ is an $e$-$p$-parabolic composition, and thus $\symgp_\tau$ is an $e$-$p$-parabolic subgroup.
\end{proof}

%% file: ms.bbl
\begin{thebibliography}{99}
\bibitem{alperin}
 {J. L. Alperin},
 {\em Local representation theory} (Cambridge Studies in Advanced Mathematics Vol. 11, Cambridge Univ. Press, 1986).

\bibitem{apostol}
{T. M. Apostol}, `Resultants of cyclotomic polynomials', {\em Proc. London Math. Soc. } 24 (1970) 457--462.

\bibitem{cartan}
{H. Cartan \and S. Eilenberg},
{\em Homological algebra } (Princeton Mathematical Series Vol. 19, Princeton Univ. Press, 1956).

\bibitem{curtisreiner1}
{C. W. Curtis \and I. Reiner},
{\em Methods of representation theory} (Wiley--Interscience Series Vol. 1, John Wiley and Sons, 1981).

\bibitem{dipperdu}
{R. Dipper \and J. Du}, `Trivial and alternating source modules of Hecke algebras', {\em Proc. London Math. Soc. } 66 (1993) 479--506.

\bibitem{djblocks}
{R. Dipper \and G. James}, `Blocks and idempotents of Hecke algebras of general linear groups', {\em Proc. London Math Soc.} 54 (1987) 57--82.

\bibitem{du}
{J. Du}, `The Green correspondence for the representation of Hecke algebras of type $A_{r-1}$', {\em Trans. Amer. Math. Soc.} 329 (1992) 273--287.

\bibitem{erdmann}
{K. Erdmann \and D. K. Nakano}, `Representation type of Hecke algebras of type $A$', {\em Trans. Amer. Math. Soc.} 354 (2002) 275--285.

\bibitem{goldschmidt}
{D. M. Goldschmidt}, {\em Group characters, symmetric functions and the Hecke algebra} (University Lecture Series Vol. 4, Amer. Math. Soc., 1993).

\bibitem{grahamlehrer}
{J. J. Graham \and G. I. Lehrer}, `Cellular algebras', {\em Invent. Math.} 123 (1996) 1--34.

\bibitem{counterexample}
{J. Hu}, `A counter--example of Dipper--Du's conjecture' {\em Arch. Math. (Basel)} 78 (2002) 449--453.

\bibitem{humphreys}
{J. E. Humphreys},  {\em Reflection groups and Coxeter groups} (Cambridge Studies in Advanced Mathematics Vol. 29, Cambridge Univ. Press, 1990).

\bibitem{jameskerber}
{G. James \and A. Kerber}, {\em The representation theory of the symmetric group} (Encyclopedia of Mathematics and its Applications Vol. 16, Addison--Wesley, 1981).

\bibitem{jones}
{L. Jones}, `Centers of generic Hecke algebras', {\em Trans. Amer. Math. Soc.} 317 (1990) 153--192.

\bibitem{mathas}
{A. Mathas}, {\em Iwahori--Hecke algebras and Schur algebras of the symmetric group} (University Lecture Series Vol. 15, Amer. Math. Soc., 1999).

\bibitem{nagell}
{T. Nagell}, `Contributions \`a la th\'eorie des corps et des polynomes cyclotomiques', {\em Ark. Mat.} 5 (1964) 153--192.

\bibitem{sagan}
{B. E. Sagan}, {\em The symmetric group: representations, combinatorial algorithms, and symmetric functions} (Graduate Texts in Mathematics Vol. 203, Springer, 2001).

\bibitem{schmider}
{S. Schmider}, {\em Hecke algebras of type A: Auslander--Reiten quivers and branching rules} (PhD Thesis, TU Kaiserslautern, 2016).
\end{thebibliography}
